\documentclass[11pt]{article}
\usepackage{amsmath,amsthm,amssymb}
\usepackage{mathtools}
\usepackage[T1]{fontenc}
\usepackage[utf8]{inputenc}
\usepackage[dvipdf]{graphicx}
\usepackage{color}
\usepackage{epstopdf}
\usepackage{dsfont}
\usepackage{enumerate}
\usepackage{enumitem}
\usepackage{bbm}

\definecolor{db}{RGB}{0, 0, 130}
\usepackage[colorlinks=true,citecolor=red,linkcolor=db,urlcolor=blue,pdfstartview=FitH]{hyperref}

\usepackage[normalem]{ulem}
\usepackage[numbers]{natbib}

\definecolor{rp}{rgb}{0.25, 0, 0.75}
\definecolor{dg}{rgb}{0, 0.6, 0}

\newtheorem{theorem}{Theorem}[section]

\newtheorem{definition}{Definition}[section]

\newtheorem{corollary}[definition]{Corollary}

\newtheorem{assumption}[theorem]{Assumption}
\newtheorem{lemma}[definition]{Lemma}

\newtheorem{proposition}[definition]{Proposition}
\newtheorem{remark}[definition]{Remark}

\def\1{\mathbbm{1}}
\def\K{\mathbb{K}}
\def\R{\mathbb{R}}

\def\E{\mathbb{E}}
\def\N{\mathbb{N}}

\def\x{\times}
\def\Om{\Omega}

\def\Fc{\mathcal{F}}
\def\F{\mathbb{F}}
\def\P{\mathbb{P}}

\def\gammab{\bar{\gamma}}

\def\eps{\varepsilon}

\def\Tc{\mathcal{T}}

\def\Gc{\mathcal{G}}
\def\Lc{\mathcal{L}}
\def\Pc{\mathcal{P}}

\def\Wc{\mathcal{W}}
\def\Mc{\mathcal{M}}

\def\Hc{\mathcal{H}}
\def\Xk{X^k}
\def\Xmuk{X^{1,k}}
\def\Xnuk{X^{2,k}}

\def\Wk{W^k}
\def\Qk{Q^k}
\def\<{\langle}
\def\>{\rangle}

\DeclareMathOperator{\Tr}{Tr}
\def\gammab{\bar \gamma}

\def\Ft{\tilde \F}

\def\Zt{\tilde Z}
\def\Ac{\mathcal{A}}

\def\As{\mathsf{A}}
\def\as{\mathsf{a}}

\def\Kt{\tilde{K}}
\def\Xt{\tilde{X}}
\def\Zt{\tilde{Z}}
\def\Pt{\tilde{\P}}

\def\ie{\textit{i.e.}}
\def\eg{\textit{e.g.}}

\DeclareMathOperator*{\argmin}{arg\,min}

\title{Optimal Control of McKean-Vlasov Branching Diffusion Processes}

\author{
	Julien Claisse \thanks{Universit\'e Paris-Dauphine, PSL University, CNRS, CEREMADE, Paris. claisse@ceremade.dauphine.fr}
	\and Jiazhi Kang\thanks{Department of Mathematics, The Chinese University of Hong Kong. jzkang@math.cuhk.edu.hk}
	\and Tianxu Lan\thanks{Department of Mathematics, The Chinese University of Hong Kong. txlan@math.cuhk.edu.hk}
        \and Xiaolu Tan\thanks{Department of Mathematics, The Chinese University of Hong Kong. xiaolu.tan@cuhk.edu.hk, Research supported
by Hong Kong RGC General Research Fund (projects 14302921).}
}

\date{\today}

\begin{document}

\maketitle

	We study an optimal control problem on McKean-Vlasov branching diffusion processes, in which the interaction term is determined by the marginal measure induced by all alive particles in the system. Accordingly, the value function is defined on the space of finite nonnegative measures over the Euclidean space.
	Within the framework of Lipschitz continuous closed-loop controls, and by using  uniqueness of solution to the associated nonlinear Fokker--Planck equation,
	we establish the dynamic programming principle.
	Further, under regularity assumptions, we show that the value function satisfies a Hamilton--Jacobi--Bellman master equation defined on the space of finite nonnegative measures.
	We next provide a corresponding verification theorem, which provides an optimal strategy under stringent assumption.
	Finally, we study a linear--quadratic example, for which explicit solutions are derived in terms of Riccati-type equations.

\section{Introduction}

	The optimal control problem for McKean-Vlasov (or mean-field) dynamics has received significant attention in recent years.
	The McKean-Vlasov stochastic differential equations (SDEs) are motivated by the modelling of the limiting behaviour of large population interacting particle systems as the population size tends to infinity (see, e.g., \cite{Mckean1967, Sznitman 1991, Dawson Vaillancourt 1995}).
	In a similar spirit, the McKean–Vlasov optimal control problems arise as the mean-field limits of optimal control problems for large population systems (see, e.g., Carmona and Delarue \cite{CarmonaDelarue}).
	For this new variation of the optimal control problem, both the dynamic programming and the maximum principle approaches have been developed (see, e.g., Carmona and Delarue \cite{CarmonaDelarue2015}, Pham and Wei \cite{Pham Wei 2018}, etc.), and explicit solutions have also been obtained in the linear–quadratic setting (see, e.g., Yong \cite{Yong2013}, etc.).
	
	\vspace{0.5em}

	While classical McKean–Vlasov dynamics typically describe populations of constant size over the time horizon, one may also incorporate population-size dynamics through branching mechanisms.
	Within the framework of stochastic differential games, a mean-field game with branching was studied by Claisse, Ren and Tan \cite{Claisse Ren Tan 2019}.
	The McKean–Vlasov dynamics combined with birth–death processes have also been investigated by Fontbona and M\'el\'eard \cite{Fontbona1}, Fontbona and Muñoz-Hernández \cite{Fontbona2}.
	More recently, in Claisse, Kang and Tan \cite{Claisse Kang Tan 2024}, we studied a more general class of McKean–Vlasov branching SDEs, establishing well-posedness results together with the propagation of chaos property.
	In another recent work Cao, Ren and Tan \cite{CaoRenTan}, a quantitative weak propagation of chaos result has also been obtained.

	\vspace{0.5em}

	For branching diffusion processes without interaction, the associated optimal control problems have also been explored in the literature.
	The study dates back to Nisio \cite{Nisio 1985}, who applied a controlled semigroup approach.
	Recently, Claisse \cite{Claisse 2018} investigated the optimal control of branching processes via a dynamic programming approach, with particular attention to the associated Hamilton–Jacobi–Bellman (Hamilton-Jacobi-Bellman) equation.
	More recently, a target control problem of the branching processes has been studied by Kharroubi and Ocello \cite{KharroubiOcello}.
	Let us also mention the paper of Hambly and Jettkant \cite{HamblyJ} which studies the maximum principle of a controlled Fokker--Planck equation problem,
	which corresponds to the branching diffusion processes.

	\vspace{0.5em}

	The main objective of this paper is to study the optimal control problem for McKean–Vlasov branching diffusion processes, building upon our well-posedness results established in \cite{Claisse Kang Tan 2024} and by using a dynamic programming approach.
	In contrast to the branching diffusion model with path-dependent coefficients studied in \cite{Claisse Kang Tan 2024}, we focus here on the Markovian setting, where the coefficients depend on both the particle position and the marginal measure induced by all alive particles.
	As a consequence, the value function of the control problem is defined as a functional on the space of finite nonnegative measures on $\R^d$.
	Although distinct particle-tree configurations may induce the same marginal finite measure on $\R^d$, we show that they nevertheless yield the same cost value by using a uniqueness condition of solutions to the associated nonlinear Fokker–Planck equation.
	We also provide sufficient conditions (and a methodology) to establish the existence and uniqueness of this nonlinear Fokker–Planck equation,
	which should be of independent interest.

	\vspace{0.5em}

	Next, following Pham and Wei \cite{Pham Wei 2018}, we introduce a closed-loop control formulation, where admissible controls are Lipschitz functionals of the particle positions.
	This framework enables the application of the well-posedness results for McKean–Vlasov branching SDEs obtained in \cite{Claisse Kang Tan 2024}, from which we deduce the dynamic programming principle (DPP).
	In our setting without common noise, the marginal measures induced by the branching diffusion process evolve as a deterministic flow of measures, so that the DPP takes a particularly tractable form.

	\vspace{0.5em}

	Under suitable smoothness assumptions on the value function, and by using a standard extension of It\^o’s formula along the flow of measures induced by branching diffusion processes, one checks that the value function provides a classical solution to a Hamilton–Jacobi–Bellman master equation on the space of finite nonnegative measures.
	Conversely, we establish a verification theorem, which ensures that the value function as well as the optimal closed-loop control can be recovered from a smooth solution to the Hamilton-Jacobi-Bellman master equation.
	Notice also that a notion of viscosity solution for this class of master equation is also developed in Ekren, He, Lan and Tan \cite{EkrenHLT}.  

	\vspace{0.5em}

	Finally, we specialize to the linear–quadratic setting and derive an explicit solution to the Hamilton-Jacobi-Bellman master equation based on a Riccati-type ordinary differential equation (ODE).
	By applying the verification theorem, we confirm that this explicit solution coincides with the value function of the controlled branching diffusion problem and obtain the corresponding optimal control.

	\vspace{0.5em}

	The remainder of the paper is organized as follows.
	In Section~\ref{sec:prel}, we provide some preliminary definitions of spaces and the associated metrics.
	Then we introduce in Section~\ref{sec:Formulation} the controlled McKean--Vlasov branching SDE and derive several a priori estimates as well as an invariance principle.
	Next, in Section~\ref{sec: control problem}, we formulate a closed-loop control problem whose value function is defined on the space of finite nonnegative measures, and then establish the dynamic programming principle and derive the corresponding Hamilton-Jacobi-Bellman master equation.
	Section~\ref{sec: LQ} focuses on the linear–quadratic case, where we obtain an explicit representation of the solution via a Riccati-type equation.
	Finally, in Section~\ref{sec:complement}, we provide some complements results on the McKean-Vlasov branching SDEs which are used in the paper.

\section{Preliminaries and Notations}
\label{sec:prel}
\subsection{Space of Measures}
\label{sec:measures}

	Let $(X, \rho)$ be a non-empty Polish space. 
	We denote by $\Pc(X)$ (resp. $\Mc(X)$) the space of all Borel probability (resp. finite nonnegative) measures on $X$.  
	For $\mu \in \Mc(X)$ and a $\mu$-integrable function $f : X \longrightarrow \R$, we denote the integral of $f$ with respect to $\mu$ by
	\begin{equation*}
	\langle \mu, f \rangle ~:= \int_X f(x)\, \mu(dx).
	\end{equation*}
	Equipped with the weak topology, $\Mc(X)$ inherits the property of being a Polish space from $X$. 
	A compatible metric is the bounded Lipschitz distance defined by
	\begin{equation*}
		\mathbf{d}_{\mathrm{BL}} \big(\mu_1, \mu_2 \big)
		~:=~
		\sup_{\varphi\in \mathrm{BL}_1(X)}\big\{\<\mu_1,\varphi\> - \<\mu_2,\varphi\>\big\},
	\end{equation*}
	where
	\begin{align*}
		\text{BL}_1(X)
		~:=~
		\Big\{
			\varphi:X\rightarrow\R~:~
			~\big|\varphi(x) - \varphi(y)\big| \leq |x-y|~\text{and}~\big|\varphi(x)\big| \le 1,
			~\mbox{for all}~ x,y\in X
		\Big\}.
	\end{align*}
	See, \eg, Bogachev~\cite[Theorem~8.3.2]{Bogachev07}.  
	
For $p \ge 1$, we denote by $\Pc_p(X)$ the space of probability measures on $X$ with finite $p$-th moment, \ie,
	\[
	\Pc_p(X) ~:=~ \Big\{ \mu \in \Pc(X) : \int_X \rho(x, x_0)^p\, \mu(dx) < +\infty \Big\},
	\]
	for some (and hence for all) fixed $x_0 \in X$.  
	It is equipped with the Wasserstein distance $\Wc_p$ defined by
	\[
	\Wc_p(\mu, \nu)
	~:=~ \inf_{\lambda \in \Lambda(\mu, \nu)} 
	\bigg( \int_{X \times X} \rho(x, y)^p\, \lambda(dx, dy) \bigg)^{1/p},
	\]
	where $\Lambda(\mu, \nu)$ is the collection of probability measures on $X \times X$ with marginals $\mu$ and $\nu$.  
	
		We denote further by $\Mc_p(X)$ the space of finite measures on $X$ with finite $p$-th moment, \ie,
	\[
		\Mc_p(X) ~:=~ \Big\{ \mu \in \Mc(X) : \int_X \rho(x, x_0)^p\, \mu(dx) < +\infty \Big\}.
	\]
	Following \cite[Appendix~B]{Claisse Ren Tan 2019}, we introduce an extension of the Wasserstein metric on $\Mc_p(X).$ 
	Let $\bar X := X \cup \{\partial\}$ be an extension of $X$ with cemetery point $\partial$.  
	Define $\rho(x, \partial) := \rho(x, x_0) + 1$ for $x \in X$, so that $(\bar X, \rho)$ is still a metric space.  
    Let us denote next, for all $\bar\mu, \bar\nu \in \Mc_p(\bar X)$ such that $\bar \mu(\bar X) = \bar\nu(\bar X),$
	\[
	\bar \Wc_{p}(\bar \mu, \bar \nu)
	~:=~ \inf_{\bar\lambda \in \bar\Lambda( \bar \mu, \bar \nu)}
	\bigg( \int_{\bar X \times \bar X} \rho(x, y)^p\, \bar\lambda(dx, dy) \bigg)^{1/p},
	\]
	where $\bar\Lambda(\bar \mu, \bar \nu)$ denotes the set of finite measures on $\bar X \times \bar X$ with marginals $\bar\mu$ and $\bar \nu$.  
	Then, for $\mu, \nu \in \Mc_p(X)$, we can define the extended Wasserstein distance by
	\[
	\Wc_p(\mu, \nu)
	~:=~ \bar\Wc_{p}(\bar\mu_m, \bar\nu_m),  \quad \text{for any } m \ge \mu(X) \vee \nu(X),
	\]
	where 
	\[
	\bar\mu_m := \mu(\cdot \cap X) + (m - \mu(X))\, \delta_\partial(\cdot),
	\qquad
	\bar\nu_m := \nu(\cdot \cap X) + (m - \nu(X))\, \delta_\partial(\cdot).
	\]
	Notice that $\Wc_p$ is independent of the constant $m$ by definition.
	Additionally, it follows by arguments identical to~\cite[Lemma~B.2]{Claisse Ren Tan 2019} that a sequence $(\mu_n)_{n\in\N}$ converges to $\mu$ in  $\Mc_p(X)$  if and only if, for all $\phi: X \to \R$ continuous satisfying $|\phi(x)| \le C(1+\rho(x_0,x)^p)$,
		\begin{equation*}
			\int_{X} {\phi(x) \,\mu_n(dx)} ~\xrightarrow[n\to\infty]{}~ \int_{X} {\phi(x) \,\mu(dx)}.
		\end{equation*}

	\subsection{Space of Particles}
	
	 To describe the genealogy of the branching process, we use the classical Ulam-Harris-Neveu notation and we denote the set of labels by
		$$
			\K:= \{\emptyset\}\cup\bigcup_{n=1}^{+\infty}\N^n.
		$$
		Given $k, k'\in\K$ with $k=k_1...k_n$ and $k'=k'_1...k'_m$, we define the concatenation  $kk':= k_1...k_nk'_1...k'_m$ and the partial order $k\prec k'$ if there exists $\tilde{k}\in\K$ such that $k'=k\tilde{k},$ which means that $k'$ is a descendant of $k.$ 
		The label $\emptyset$ corresponds to the common ancestor of the population akin to the root of the tree.
		
		Next we introduce the state space of branching diffusion processes as
	\[
	E 
	~:=~
	\bigg\{
	\sum_{k \in K} \delta_{(k, x^k)} : K\subset \K \text{ finite},~
	x^k \in \R^d, ~ k \nprec k',~
	\text{for all } k, k' \in K
	\bigg\},
	\]
	where the Dirac measure $\delta_{(k, x^k)}$ corresponds to a particle identified by a label $k$ and a position $x^k.$
	Motivated by the study of superprocesses, it is now common to represent branching diffusions as measure-valued processes, see, \eg, Dawson \cite{Dawson 1993}.
  For a measurable function $f = (f^k)_{k \in \K} : \K \times \R^d \to \R,$ we observe that
	\[
	\langle e, f \rangle
	~=~
	\sum_{k \in K} f^k(x^k), \quad \text{for all } e = \sum_{k \in K} \delta_{(k, x^k)} \in E.
	\]
Notice that the space $E$ is a closed subset of $\Mc(\K \times \R^d)$ under the topology of weak convergence, and thus $E$ is also a Polish space. 
	
	We then introduce a metric $d_E$ on $E$ compatible with the weak  topology:  for any $e_1, e_2 \in E$ such that $e_1 = \sum_{k \in K_1} \delta_{(k, x^k)}$ and $e_2 = \sum_{k \in K_2} \delta_{(k, y^k)},$
	\begin{equation} \label{eq:def_d_E}
	d_E(e_1, e_2)
	~:=~
	\#(K_1 \triangle K_2)
	~+~
	\sum_{k \in K_1 \cap K_2} \big(|x^k - y^k| \wedge 1 \big) ,
	\end{equation}
	where $\#(K_1 \triangle K_2)$ denotes the cardinality of $K_1 \triangle K_2:= (K_1 \setminus K_2) \cup (K_2 \setminus K_1).$ 
	The corresponding space of probability distribution on $E$ with finite first moment is given by 
		\begin{equation*}
			\Pc_1(E) 
			=
			\big\{ 
				\nu \in \Pc(E) ~:  \int_{E} \langle e,1{}\rangle \,\nu(d e) < \infty 
			\big\}.	
		\end{equation*}
	This is consistent with the definition of Section~\ref{sec:measures} since $d_E(e,e_0) = \<e,1\>$ if $e_0$ denotes the null element of $E.$
	
	We conclude this section by introducing a key projection mapping which embeds $\Pc_1 (E)$ into $\Mc(\R^d)$ by detaching the label of particles.  
	Namely, we define $\pi:\Pc_1 (E) \longrightarrow \Mc(\R^d)$ as follows: for all $\nu \in \Pc_1 (E)$,
	\begin{equation}\label{definition h}
		\big\langle \pi(\nu), \varphi \big \rangle 
		:= \int_E \< e, \varphi \>\,\nu(de),
		~~\mbox{for all}~\varphi\in C_b(\R^d).
	\end{equation}
	It is a Lipschitz continuous map as established below.

\begin{lemma}\label{lemma continuity}
	It holds for all $\nu_1, \nu_2\in\Pc_1(E),$
	\begin{align*}
		\mathbf{d}_{\mathrm{BL}}\big(\pi(\nu_1), \pi(\nu_2)\big)
		~\leq~
		2 \Wc_1(\nu_1, \nu_2).
	\end{align*}
\end{lemma}

\begin{proof}
	Let $Z_1 = \sum_{k\in K_1}\delta_{(k,X^k)}$ and $Z_2 = \sum_{k\in K_2}\delta_{(k,Y^k)}$ be an arbitrary couple of $E$-valued random variables such that $\nu_1 = \Lc(Z_1)$ and $\nu_2 = \Lc(Z_2).$ 
		Recall that
	\begin{align*}
		\mathbf{d}_{\mathrm{BL}} \big(\pi(\nu_1), \pi(\nu_2)\big)
		=
		\sup_{\varphi\in \text{BL}_1(\R^d)}\big\{\<\pi(\nu_1),\varphi\> - \<\pi(\nu_2),\varphi\>\big\},
	\end{align*}
	where $\mathrm{BL}_1(\R^d)$ is the space of all Lipschitz functions on $\R^d$ with Lipschitz constant $1$ and uniformly bounded by $1$.
	Then a straightforward computation yields that, for any $\varphi\in \text{BL}_1(\R^d)$,
	\begin{align*}
	 \big|\<\pi(\nu_1),\varphi\> - \<\pi(\nu_2),\varphi\> \big|
		& =~
		\Big| \E\Big[\sum_{k\in K_1} \varphi (X^k)\Big] - \E\Big[\sum_{k\in K_2} \varphi(Y^k)\Big]  \Big|\\
		& \le~
			\E\Big[\sum_{k\in K_1\cap K_2}\big|\varphi (X^k) - \varphi(Y^k)\big| +\sum_{k\in K_1\setminus K_2}\big|\varphi(X^k)\big| + \sum_{k\in K_2\setminus K_1}\big|\varphi(Y^k)\big|\Big]\\
		& \leq~
			\E\Big[\sum_{k\in K_1\cap K_2}|X^k - Y^k|\wedge 2 + \# (K_1\triangle K_2)\Big].
	\end{align*}
    By definition of $d_E$ in~\eqref{eq:def_d_E}, we deduce that
	\begin{align*}
		\mathbf{d}_{\mathrm{BL}} \big(\pi(\nu_1), \pi(\nu_2)\big)
		~\leq~
		2\, \E\big[d_E(Z_1, Z_2)\big].
	\end{align*}
	The conclusion follows by taking the infimum over all couples $(Z_1,Z_2).$
\end{proof}

\section{Controlled McKean-Vlasov Branching Diffusion}
\label{sec:Formulation}

	Let us first introduce the class of controlled McKean-Vlasov branching diffusion processes, 
	and then derive some key properties, including a priori estimates and an invariance principle.
	Throughout the paper, we fix a subset $A \subset \R^n$ which serves as the action space.

\subsection{Construction and SDE Formulation} 

	The coefficients of the McKean-Vlasov branching diffusion process are given by
	\[
	\big( b, \sigma, \gamma, (p_{\ell})_{\ell \in \N} \big)
	:
	[0,T] \times \R^d \times \Mc(\R^d) \times A
	~\longrightarrow~
	\R^d \times \R^{d \times d} \times [0, \bar{\gamma}] \times [0,1]^{\N},
	\]
	where $\bar{\gamma} > 0$ is a fixed constant.
	Namely, $b$ and $\sigma$ are the drift and diffusion coefficients for the dynamic of each particle, $\gamma$ is the death rate, and $(p_{\ell})_{\ell \in \N}$ is the probability mass function of the progeny distribution. 
	In particular, it holds that $p_{\ell}(\cdot) \in [0,1]$ for each $\ell \in \N$, and $\sum_{\ell \in \N} p_{\ell}(\cdot) = 1$.
	Let us also define a partition $(I_{\ell}(\cdot))_{\ell \in \N}$ of $[0,1]$ by
$$
	I_{\ell} (\cdot)
	~:=~
	\Big[
		\sum_{i=0}^{\ell -1}p_i (\cdot), ~\sum_{i=0}^{\ell} p_i (\cdot) 
	\Big),
	~~\mbox{for each}~
	\ell \in \N.
$$

	Let $(\Om, \Fc, \F, \P)$ be a filtered probability space, with filtration $\F = (\Fc_t)_{t \ge 0}$ satisfying the usual conditions, 
	equipped with a family of mutually independent $d$-dimensional Brownian motions $(\Wk)_{k\in\K}$ and Poisson random measures $(\Qk(ds,dz))_{k\in\K}$ on $[0,T] \x [0, \gammab] \x [0,1]$ with Lebesgue intensity measure $ds\x dz$.
	It is assumed further that the $\sigma$--field $\Fc_0$ is sufficiently large to support a random variable for any distribution on $E$ and we denote 
	\begin{equation*}
	\Xi_t 
		  ~:=~
		  \big\{ \xi: \Om\to E  ~\Fc_t\mbox{-measurable such that } \E[\<\xi, 1\>]<+\infty \big\}.
	\end{equation*}

	Now
	we describe the dynamic of a branching diffusion process starting at time $t\in [0,T]$ in state $\xi\in\Xi_t$ and being controlled by $(\alpha^k)_{k \in \K}$ a family of $A$-valued predictable processes.
	It is represented as a $E$-valued process $(Z_s)_{s \in [t,T]}$ of the form
	\begin{equation*} \label{eq:Xk2Z}
		Z_s ~:=~ \sum_{k\in K_s}\delta_{(k,\Xk_s)}, ~~s \in [t,T],
	\end{equation*}
	where $K_s$ denotes the collection of labels of all particles alive at time $s \in [t,T]$ and $\Xk_s$ denotes the position of particle $k\in K_s$.
    We start from a prescribed initial condition $Z_t=\xi.$ Then the dynamic of each particle $k\in K_s$ is given by the controlled SDE
	\begin{equation}\label{eq:SDE}
		d\Xk_s ~=~ b(s, \Xk_s,\mu_s, \alpha^k_s)\,ds+\sigma(s, \Xk_s,\mu_s,\alpha^k_s)\,d\Wk_s,
	\end{equation}
	where $\mu_s\in\Mc(\R^d)$ corresponds to the mean-field interaction term, also called marginal measure, 
    defined by
	\begin{equation}\label{eq:def_mu_t}
	\<\mu_s, \varphi\> 
	~:= ~
	\E\Big[
		\sum_{k\in K_s}\varphi(\Xk_s)	
	\Big],
	~~\mbox{for all}~
	\varphi\in C_b(\R^d).
	\end{equation}
	Denote further by $S_k$ the birth time of particle $k$ and fix $S_k = t$ for each initial particle $k \in K_t$.
	Then each particle $k\in K_s$ runs a death clock with intensity $\gamma(s, \Xk_s, \mu_s,\alpha^k_s)$, \ie, it dies at time
$$
	T_k
	~:=~
	\inf
	\big\{
		s>S_k: \Qk \big(\{s\}\x[0,\gamma(s,\Xk_s,\mu_s,\alpha^k_s)] \x[0,1] \big) =1
	\big\}.
$$
Let $U_k$ be the unique random variable, uniformly distributed over the interval $[0,1],$ such that
$$
	\Qk \big(\{T_k\} \x [0,\gamma(T_k,\Xk_{T_k},\mu_{T_k},\alpha^k_{T_k})] \x \{U_k\} \big) ~=~1.
$$
When $U_k$ falls into the set $I_{\ell}(T_k, \Xk_{T_k}, \mu_{T_k}, \alpha_{T_k})$ of the partition of $[0,1]$, 
the particle $k$ gives birth to $\ell$ offspring particles labelled by $k1, \cdots, k \ell$ and we set
\begin{align*}
	K_{T_k} ~=~ \left(K_{T_k-} \setminus \{k\}\right) \cup \{k1, \cdots,k \ell\}.
\end{align*}
Then the birth time of the offspring particles is defined as the death time of the parent particle, \ie, $S_{ki}:= T_k$ for $i = 1, \cdots , \ell$.
Further, we consider that the offspring particles start from the same position as the parent particle, \ie,
$$
	X^{ki}_{S_{ki}} ~=~ X^k_{T_k-}, ~~\text{for } i ~=~ 1, \cdots, \ell.
$$
This completes the first presentation of a controlled McKean-Vlasov branching diffusion process by induction.

	Let us also provide a more formal definition of the process described above by means of SDE.
	Denote by $\Lc$ the infinitesimal generator of the diffusion $(b,\sigma)$, \ie, for all $(s, x ,m,a)\in[0,T]\x\R^d\x\Mc(\R^d)\x A$ and $\varphi \in C^2_b(\R^d,\R)$,
	\begin{equation*} \label{eq:def_Lc}
		\Lc \varphi (s, x, m, a)
		~:=~
		\frac{1}{2}\Tr\big(\sigma\sigma^{\top} (s, x,m, a)\nabla^2_x \varphi( x)\big) + b(s,x,m, a)\cdot \nabla_x \varphi(x),
	\end{equation*}
	where we denote by $\nabla_x, \nabla^2_x$ the gradient and Hessian operators acting on the space variables respectively.
	Then the McKean-Vlasov branching diffusion controlled by $(\alpha^k)_{k\in\K}$ with initial state $\xi$ at time $t$ can be characterized as the solution to the following SDE: for all $f:=(f^k)_{k\in\K}\in C^2_b (\K\x\R^d, \R),$ $s\in[t,T],$
\begin{multline} \label{eq:SDE_McKean-VlasovB}
	\< Z_s, f \>
	= 
	\<\xi ,f\>
	+
	\int_t^s \sum_{k\in K_\theta}\Lc f^k(\theta, \Xk_\theta, \mu_\theta,\alpha^k_\theta) \,d\theta
	+
	\int_t^s \sum_{k\in K_\theta}
	\nabla_x f^k(\Xk_\theta)\sigma(\theta,\Xk_\theta,\mu_\theta,\alpha^k_\theta)
	\,d\Wk_\theta \\		
	+\int_{(t,s]\x [0,\gammab] \x [0,1]}
	\sum_{k\in K_{\theta-}}\sum_{ \ell \geq 0}
	\Big(\sum_{i=1}^{ \ell} 
	f^{ki} -f^k\Big)
	(\Xk_\theta)\1_{[0,\gamma(\theta,\Xk_\theta,\mu_\theta,\alpha^k_\theta)]\x I_{\ell}(\theta,\Xk_\theta,\mu_\theta,\alpha^k_\theta)}(z)\,\Qk( d\theta,dz).
\end{multline}

	In the next section, we show that, under suitable assumptions, the controlled McKean-Vlasov branching diffusion process introduced above is well-defined for a class of closed-loop controls.

\subsection{Existence and Uniqueness}

	The coefficient functions will be assumed to satisfy the following conditions in the sequel. 
	This ensures, in particular, that the process is well-defined for Lipschitz continuous closed--loop controls as established in Proposition~\ref{proposition existence and uniqueness} below.

	\begin{assumption}\label{A.1}

 \noindent $\mathrm{(i)}$ The death rate $\gamma(\cdot)$ is bounded by the constant $\gammab>0$ and the mean of the progeny distribution  $\sum_{ \ell \ge 0} \ell p_{\ell}(\cdot)$ is bounded by a constant $M_1>0.$ 
In addition, there exists a probability mass function $(\bar{p}_\ell)_{\ell\ge 1}$ such that $\sum_{\ell\ge 1} {\ell\bar{p}_\ell} <+\infty$ and  $p_\ell(\cdot)\le C \bar{p}_\ell$ for all $n\ge 0,\ell\ge 1,$ for some $C>0.$ 
	
	\noindent $\mathrm{(ii)}$ The coefficient functions $b,$ $\sigma,$ $\gamma,$ $(p_\ell)_{\ell\in\N}$ are Lipschitz in $(x,m,a)$ in the sense that, for each $\varphi = b, \sigma, \gamma, p_\ell,$ there exists $C_\varphi>0$ such that 
	\begin{align*}
		\big| \varphi(s, x ,m, a) - \varphi(s, x',m', a') \big|   
		~&\leq~
		C_\varphi
		\big( |x-x'| + \mathbf{d}_{\mathrm{BL}}(m, m')  + |a-a'|\big), 
	\end{align*}
	for all $(s, x, x', m, m', a, a') \in [0,T] \x \R^d \x \R^d \x \Mc(\R^d) \x \Mc(\R^d) \x A \x A.$ 
	In addition, the Lipschitz constants $(C_{p_\ell})_{\ell\in\N}$ of the progeny distribution satisfy $\sum_{\ell \in\N} \ell C_{p_\ell} < \infty.$
	
	\noindent $\mathrm{(iii)}$ The volatility $\sigma$ is bounded and the drift $b$ satisfies a linear growth condition in $(x,a)$ in the sense that there exists $C_b>0$ such that 
	\begin{align*}
		\big|b(s, x ,m, a)\big|
		~&\leq~
		C_b\big(1+|x|  + |a|\big),
	\end{align*}
	for all $(s, x, m, a) \in [0,T] \x \R^d \x \Mc(\R^d) \x A.$
	\end{assumption}
	
Let us introduce next the class of admissible controls which correspond to Lipschitz continuous closed-loop strategies in our setting.

\begin{definition}\label{def:admissible}
	The class of admissible controls $\Ac$ is the collection of control processes $(\alpha^k)_{k \in \K}$ of the form
	$\alpha^k_s = \alpha(s, \Xk_s,\mu_s)$ for a mapping $\alpha:[0,T]\x\R^d\x\Mc(\R^d) \longrightarrow A$
	satisfying the following conditions: there exists $C_{\alpha}>0$ such that
	\begin{align*}
		\big|\alpha(s,x,m) - \alpha(s,x,m')\big|
		~&\leq~
		C_{\alpha} \left(|x-x'| +\mathbf{d}_{\mathrm{BL}}(m, m') \right),  \\
		\big|\alpha(s,x,m)\big|
		~&\leq~
		C_{{\alpha}} \big(1 + | x | \big),		
	\end{align*} 
	for all $(s,x,x',m,m')\in[0,T]\x\R^d\x\R^d\x\Mc(\R^d)\x\Mc(\R^d).$ 
\end{definition}

We can then introduce the following notation: for any map $\varphi$ on $[0,T]\x\R^d\x\Mc(\R^d)\x A$ and control $\alpha\in\Ac,$ 
\begin{equation*}
 \varphi^{\alpha} (s,x,m) :=  \varphi \big(s,x,m,\alpha(s,x,m)\big), \quad \text{for all }(s,x,m)\in [0,T]\x\R^d\x\Mc(\R^d).
\end{equation*}
Observe that one can think of a McKean-Vlasov branching diffusion process controlled by a closed-loop control $\alpha$ as an uncontrolled process with coefficient $(b^\alpha,\sigma^\alpha,\gamma^\alpha,p^\alpha_\ell).$ In particular, the existence and uniqueness result below relies on this observation.  So does the invariance principle in the next section.

\begin{proposition}\label{proposition existence and uniqueness}
	Let Assumption \ref{A.1} hold. Let also $t\in [0,T],$ $\xi\in \Xi_t$ 
	and $\alpha\in\Ac.$
	Then there exists a unique (up to indistinguishability) $E$-valued adapted càdlàg process 
    	\begin{equation*}
		Z^{t,\xi,\alpha}_s = \sum_{k\in K^{t,\xi,\alpha}_s}\delta_{(k,X^{k}_s)}, \quad s\in [t,T],
	\end{equation*}
	satistyfing \eqref{eq:SDE_McKean-VlasovB} together with the McKean-Vlasov condition \eqref{eq:def_mu_t}, such that
	\begin{equation}\label{ineq:supK}
		\E\Big[\sup_{s\in[t,T]} \# K^{t,\xi,\alpha}_s\Big]
		~\leq~
		\E\big[\<\xi, 1\>\big] e^{\bar\gamma M_1 (T-t)}.
	\end{equation}	
\end{proposition}

\begin{proof}
	It follows directly from Proposition~\ref{proposition existence and uniqueness-bis} below which is a slight extension of the strong existence and uniqueness result for (uncontrolled) McKean-Vlasov branching diffusion in~\cite[Theorem~2.3]{Claisse Kang Tan 2024}.
	Indeed,  using the Lipschitz continuity of $\pi$ from Lemma \ref{lemma continuity} together with Assumption~\ref{A.1} and Definition~\ref{def:admissible},
	we can easily check that, for each $\alpha\in\Ac,$ the coefficient functions 
	\begin{equation*}
	(s,x,\nu)\in [0,T]\x\R^d\x\Pc_1(E)  \longmapsto \big(b^\alpha, \sigma^\alpha, \gamma^{\alpha}, p^{\alpha}_{\ell} \big)\big(s, x, \pi(\nu)\big),
	\end{equation*}
	 satisfy the conditions required in Assumption~\ref{A.1-bis} below. In particular, they are Lipschitz continuous $(x,\nu),$ the drift term satisfying further a linear growth condition in $x.$
\end{proof}

We next provide an a priori estimate on the second moment of the McKean-Vlasov branching diffusion process.

\begin{lemma}\label{lemm:upperbound}
 Under the conditions of Proposition~\ref{proposition existence and uniqueness}, if we assume further that $\E[\<\xi, |\cdot|^2\>] < \infty,$ then there exists $C>0$ such that 
	\begin{align*}
		\E\bigg[\sup_{s\in [t,T]} \sum_{k\in K^{t,\xi,\alpha}_s}\big(1+|X^{k}_s|^2\big)\bigg] \le C \E\big[\<\xi, 1+|\cdot|^2\>\big].
	\end{align*}
\end{lemma}

\begin{proof}
 	It follows directly from Lemma~\ref{lemm:upperbound-bis} below, stated in the uncontrolled setting.
\end{proof}

\subsection{Invariance Principle}\label{sec:marginal}

   Let us denote by $\mu^{t,\xi, \alpha}$ the marginal measure associated to the McKean-Vlasov branching diffusion process $Z^{t,\xi,\alpha},$ \ie, 
\begin{align*}\label{eq:mean measure xi}
	\<\mu^{t,\xi, \alpha}_s,\varphi\>
	:=
	\mathbb{E}\bigg[\sum_{k \in K^{t,\xi,\alpha}_s} \varphi(X^{k}_s)\bigg],\quad \varphi\in C_b(\R^d).
\end{align*}
Observe that $\mu^{t,\xi, \alpha}_s = \pi(\Lc(Z^{t,\xi,\alpha}_s))$ where $\pi$ is defined by~\eqref{definition h}.
 Given $m\in\Mc(\R^d),$  we further denote~by
	    $$
		  \Xi_t^m 
		  ~:=~
		  \big\{ \xi:\Om\to E ~\Fc_t\mbox{-measurable such that } \pi(\Lc(\xi)) = m \big\}.
	   $$	
	   In this section, we restrict to the case $m\in\Mc_2(\R^d),$ \ie, the initial state $\xi\in\Xi_t$ satisfies the condition $\E[\<\xi, |\cdot|^2\>] < \infty.$ In view of Lemma~\ref{lemm:upperbound}, it follows that $\mu_s^{t,\xi,\alpha}\in\Mc_2(\R^d)$ for all $s\in[t,T].$
	   
	   Let us start with a preliminary lemma stating the continuity of the flow of marginal measures.
	   
	   \begin{lemma}\label{lem:flow-continuity}
 		Let Assumption \ref{A.1} hold. 
		Let also $t\in [0,T],$ $\xi\in\Xi_t$ such that $\E[\<\xi, |\cdot|^2\>] < \infty$  and $\alpha\in\Ac.$ 
		Then the flow of marginal measures $s\in [t,T] \mapsto \mu^{t,\xi,\alpha}_s\in\Mc_2(\R^d)$ is continuous.  
\end{lemma}

\begin{proof}
Recall that, by definition of the distance $\Wc_2$ on $\Mc_2(\R^d)$ in Section~\ref{sec:measures}, the continuity of $s\mapsto \mu^{t,\xi,\alpha}_s$ is equivalent to the continuity of
	\begin{equation*}
	 s\mapsto \langle\mu^{t,\xi,\alpha}_s,\phi\rangle = \E\bigg[\sum_{k\in K^{t,\xi,\alpha}_s} \phi(X^k_s)\bigg],
	\end{equation*}
   for all $\phi:\R^d\to\R$ continuous with quadratic growth. The latter follows directly by dominated convergence using Lemma~\ref{lemm:upperbound}.  
   Notice that, although the process $Z^{t,\xi,\alpha}$ has jumps,  it holds that  $\P(Z^{t,\xi,\alpha}_s = Z^{t,\xi,\alpha}_{s-}) = 1$ for all $s \in [t,T]$ since the jumps are generated by Poisson random measures.
\end{proof}

	Then we establish an invariance principle for McKean-Vlasov branching diffusion which plays a crucial role in the analysis of the control problem below. 
	It can be understood as a sort of uniqueness in marginal law for SDE~\eqref{eq:SDE_McKean-VlasovB}.
	It exploits further the symmetry of branching processes with respect to the choice of labelling.

  	\begin{proposition} \label{lemm:uniqueFKP}
		Let Assumption \ref{A.1} hold. Let also $t\in [0,T],$ $m \in \Mc_2(\R^d)$ and  $\alpha \in \Ac$.
		Then it holds
		$$
			\mu^{t,\xi_1, \alpha}_s = \mu^{t,\xi_2, \alpha}_s,
			\quad\mbox{for all }s\in[t,T],\,
			\xi_1, \xi_2 \in \Xi_t^m.
		$$
	\end{proposition}
	
	\begin{proof}
	 The proof is divided in two parts. First, we establish the result under strong assumptions on the coefficients by considering uniqueness of the corresponding Fokker-Planck equation.  Then we exploit the stability of solutions to SDE~\eqref{eq:SDE_McKean-VlasovB} to show that we can relax the assumptions on the coefficients.  For ease of notation, we assume that $t=0$ and we ignore the superscripts, including control term $\alpha,$ in the notations.   
	 Let us also denote the growth rate of the population by
	$$
		\kappa(t,x,m) := \gamma(t,x,m) \sum_{\ell \ge 0} (\ell -1) p_{\ell}(t,x,m), 
		~~\mbox{for all}~(t,x,m) \in [0,T] \x \R^d \x \Mc_2(\R^d).
	$$
	
\noindent\emph{Step 1.} 
   Consider the following nonlinear Fokker--Planck equation
	\begin{equation}\label{eq:FKP}
		\partial_t \mu + \sum_{i=1}^d {\partial_{x_i}\big( b(t,x,\mu_t) \mu \big)} - \frac12 \sum_{i,j=1}^d {\partial^2_{x_i x_j} \big( [\sigma\sigma^{\top}]_{i,j}(t,x, \mu_t) \mu \big)} + \kappa(t,x,\mu_t) \mu = 0,
	\end{equation}	
	with initial condition $\mu_0 = m \in \Mc_2(\R^d).$
We observe first that, for any initial state $\xi\in\Xi_0^m,$ the flow of marginal measures~\eqref{eq:def_mu_t} induces a distributional solution to this equation in the sense that for any test function $\varphi\in C^2_b([0,T]\x\R^d),$
	\begin{align*}
		&\<\mu_T,\varphi(T,\cdot)\> 
		~=~
		\<m, \varphi(0,\cdot)\> +
		\int_0^T
		\big\<\mu_t, \partial_t \varphi(t, \cdot)
		+
		\Lc\varphi(t, \cdot,\mu_t) 
		+
		\kappa \big(t, \cdot, \mu_t\big)\varphi(t, \cdot)\big\>
		\, dt.
	\end{align*}
    Indeed, it suffices to take expectation in a variant of It\^o's formula~\eqref{eq:SDE_McKean-VlasovB} for time-dependent test functions with $f^k=\varphi$ independent of the label. Thus the desired result would follow immediately from uniqueness to the nonlinear Fokker--Planck equation~\eqref{eq:FKP}. 
    
    To prove it, let $(\mu_t)_{t \ge 0}$ be a fixed distributional solution to~\eqref{eq:FKP} and 
 consider a classical (without mean-field interaction) branching diffusion process $ \hat Z = (\hat Z_t)_{t \ge 0}$ associated to the coefficient functions $(t,x) \longmapsto \big(b, \sigma, \gamma, (p_{\ell})_{\ell \ge 0} \big) (t,x, \mu_t)$ and initial state $\xi\in\Xi_0^m$. 
	It follows easily from It\^o's formula as above that the flow of marginal measures $(\hat{\mu}_t)_{t\ge 0}$ induced by $\hat Z$ satisfies in the distributional sense
	\begin{equation}\label{eq:LFKP}
		\partial_t \hat{\mu} + \sum_{i=1}^d {\partial_{x_i}\big( b(t,x,\mu_t) \hat{\mu} \big)} - \frac12 \sum_{i,j=1}^d {\partial^2_{x_i x_j} \big( [\sigma\sigma^{\top}]_{i,j}(t,x, \mu_t) \hat{\mu} \big)} + \kappa(t,x,\mu_t) \hat{\mu} = 0,
	\end{equation}
	with initial condition $\hat{\mu}_0 = m$. 
	Now, by Hambly and Jettkant~\cite[Proposition 2.4]{HamblyJ}, the linear Fokker-Planck equation~\eqref{eq:LFKP} has a unique (continuous) distributional solution under the following additional assumptions: $b, \sigma$ are uniformly bounded,  $\sigma$ is uniformly elliptic and the initial marginal measure $m$ admits a density function satisfying for some $\eta_0>0,$
		\begin{equation} \label{eq:cond_nu}
			 \int_{\R^d} m(x)^2 \exp{\big(\eta_0 \sqrt{1+|x|^2}\big)} \,dx < +\infty.
		\end{equation}
	Thus, under the above assumptions,  we deduce that $\mu_t = \hat \mu_t$ for all $t\in[0,T].$
	This implies further that $\hat Z$ is actually a solution to the McKean-Vlasov branching diffusion SDE~\eqref{eq:SDE_McKean-VlasovB}.
	Since pathwise (and thus weak)  uniqueness holds for this SDE by~\cite[Theorem 2.3]{Claisse Kang Tan 2024}, we conclude that $(\hat{\mu}_t)_{t\ge 0}=(\mu_t)_{t\ge 0}$ is the unique distributional solution to the nonlinear Fokker--Planck equation~\eqref{eq:FKP}.
	
	\noindent \noindent\emph{Step 2.} 	For any $\varepsilon>0$ and $\xi = \sum_{k \in K_0} \delta_{X^k_0} \in \Xi_0^m,$ define $\xi^{\eps} := \sum_{k \in K_0} \delta_{X^{k,\eps}_0}$ with
	$$
		X^{k,\eps}_0 := \Big( \frac{1}{-\eps} \vee X^k_0 \wedge \frac{1}{\eps} \Big) + \eps G^k, 
	$$
	where $(G^k)_{k\in\K}$ is a sequence of i.i.d.\@ random vectors with standard normal distribution, independent of $\xi$. It is clear that $m^{\eps} := \pi (\Lc(\xi^{\eps}))$ with $\pi$ defined in \eqref{definition h} satisfies the technical conditions in~\eqref{eq:cond_nu}. 
	Consider further the following modified drift and diffusion coefficients:
	$$
		b_{\eps}(\cdot) := \frac{1}{-\eps} \vee b(\cdot) \wedge \frac{1}{\eps}
		~~\mbox{and}~~
		 \sigma_{\eps} \sigma_{\eps}^{\top} (\cdot) := \sigma \sigma^{\top}(\cdot) + \eps I_d.
	$$
	Beware that $\sigma_{\eps} $ is not uniquely defined by the equation above.  Here we consider any version of $\sigma_\eps$ which is Lipschiptz continuous like $\sigma$, see, \eg, Stroock and Varadhan~\cite[Theorem~5.2.2]{Stroock06}.  
	
	Now consider two initial states $\xi_1$ and $\xi_2$  in $\Xi_0^m,$ and denote $Z^{\xi_1,\eps}$ (resp. $\mu^{\xi_1,\eps}$) and $Z^{ \xi_2, \eps}$ (resp. $\mu^{\xi_2,\eps}$) the solution to SDE~\eqref{eq:SDE_McKean-VlasovB} (resp. the marginal measure~\eqref{eq:def_mu_t}) associated to the coefficient $(b_{\eps}, \sigma_{\eps}, \gamma,(p_\ell)_{\ell\in\N})$ and the initial condition $\xi^{\eps}_1$ and $\xi^{\eps}_2$. Applying Step~1 above, we have by weak uniqueness of Fokker-Plank equation~\eqref{eq:FKP} that 
	\begin{equation}\label{eq:invariance-eps}
		\mu^{\xi_1,\eps}_t   = \mu^{\xi_2,\eps}_t, \qquad \text{for all }t\in[0,T].
	\end{equation}
	Further, since $(b_{\eps}, \sigma_{\eps}\sigma_{\eps}^\top)$ converges pointwise to $(b,\sigma\sigma^\top)$ and $m^\eps$ converges weakly to $m,$ the stability of the martingale problem 
     induces the weak convergence of $\mu^{\xi_i,\eps}_t$ toward $\mu^{\xi_i}_t$ in $\Mc(\R^d)$ for almost all $t\in[0,T]$ for $i=1,2.$ This is established precisely in Corollary~\ref{cor:stability} below in the uncontrolled setting.
	 The conclusion follows by passing to the limit $\varepsilon\to 0$ in~\eqref{eq:invariance-eps} for almost all $t\in[0,T]$ and using the continuity of the flow $t\mapsto \mu^{\xi_i}_t$ for $i=1,2$ established in Lemma~\ref{lem:flow-continuity}.
	\end{proof}

	Now given an arbitrary measure $m\in \Mc_2(\R^d),$ we can define in view of Proposition~\ref{lemm:uniqueFKP}, 
	\begin{equation}\label{def:flow}
	\mu^{t,m, \alpha}_s := \mu^{t,\xi, \alpha}_s,\qquad \text{for all } s\in[t,T],\,\xi \in \Xi_t^m.
 	\end{equation}
 	This is actually the (deterministic) process of interest in the optimal control problem studied in the next section.

\section{Optimal  Control Problem}\label{sec: control problem}

	We now introduce the finite horizon control problem on McKean-Vlasov branching diffusion studied in this paper.
	 It is inspired by the problem investigated in Pham and Wei~\cite{Pham Wei 2018} in the standard McKean-Vlasov setting.

\subsection{Formulation of the Problem}

	Let us introduce first the running and terminal cost functions, and make appropriate assumptions to ensure that the optimal control problem is well-defined.  

\begin{assumption}\label{A.3}
    The functions $L: [0,T] \times \mathbb{R}^d \times \Mc_2(\R^d) \times A \rightarrow \R$ and $g: \mathbb{R}^d \times \Mc_2(\R^d) \rightarrow \mathbb{R}$ are continuous 
and satisfy a quadratic growth condition in $(x,a)$ in the sense that there exist constants $C_L, C_g>0$ such that
	\begin{align*}
		\big|L(s, x ,m, a) \big| 
		&~\leq~
		C_L
		\big(1+| x |^2 + | a |^2 \big),\\
		\big|g( x ,m) \big| 
		&~\leq~
		C_g
		\big(1+| x |^2 \big),
	\end{align*}
	for all $(s,x,m,a)\in [0,T]\x\R^d\x\Mc_2(\R^d)\x A.$
\end{assumption}

	 We aim to study the following optimal control problem
	\begin{equation*}
	\inf_{\alpha\in\Ac} \E\bigg[\int_t^T 
		\sum_{k\in K^{t,\xi,\alpha}_s}L^{\alpha}\big(s,X^{k}_s,\mu^{t,\xi, \alpha}_s\big)	
		\,ds
	+
	\sum_{k\in K^{t,\xi,\alpha}_T}g\big(X^{k}_T,\mu^{t,\xi, \alpha}_T\big)
	\bigg],
	\end{equation*}
	or equivalently,
	\begin{equation*}
	\inf_{\alpha\in\Ac} \bigg\{\int_t^T \big\<\mu^{t,m, \alpha}_s, L^{\alpha} (s, \cdot, \mu^{t,m, \alpha}_s)\> \,ds
	+
	\<\mu^{t,m, \alpha}_T, g(\cdot, \mu^{t,m, \alpha}_T)\>\bigg\}.
	\end{equation*}
	where $m=\pi(\Lc(\xi))$  and $\mu^{t,m, \alpha} = \mu^{t,\xi, \alpha}$ as defined in~\eqref{def:flow}.
	Considering the second formulation coming from the invariance principle in Proposition~\ref{lemm:uniqueFKP}, it turns out that this optimization problem depends only on the deterministic flow of marginal measures  $(\mu^{t,m, \alpha}_s)_{s\in[t,T]}$ rather than on the whole stochastic process $(Z^{t,\xi, \alpha}_s)_{s\in[t,T]}.$ This key observation leads to a major simplification as it reduces the analysis to a deterministic control problem on the space of finite measure.
	
	Thus we can introduce the cost function $J:[0,T]\x\Mc_2(\R^d)\x\Ac\to\R$ as
	\begin{equation}\label{eq:cost}
	 J(t,m,\alpha) = \int_t^T \big\<\mu^{t,m, \alpha}_s, L^{\alpha}(s, \cdot, \mu^{t,m, \alpha}_s)\> \,ds
	+
	\big\<\mu^{t,m, \alpha}_T, g(\cdot, \mu^{t,m, \alpha}_T)\big>,
	\end{equation}
	and the value function $v:[0,T]\x\Mc_2(\R^d)\to\R$ as
	\begin{equation}\label{eq:value function}
	v(t,m) = \inf_{\alpha \in \Ac} J(t, m, \alpha).
	\end{equation}

\begin{proposition}\label{wellposedness of control problem}
	Let Assumptions \ref{A.1} and \ref{A.3} hold. Then the cost function $J$ as defined in \eqref{eq:cost} is well-defined and finite. 
\end{proposition}

\begin{proof}
	Fix $t\in[0,T],$ $m\in\Mc_2(\R^d), $ $\alpha\in\Ac,$ and pick any $\xi \in \Xi_t^m$ so that $\mu^{t,m, \alpha} = \mu^{t,\xi, \alpha}.$  For ease of notation, we omit the superscripts $t,\xi,m,\alpha$ in the proof.  
	Then it follows that the term corresponding to the running cost satisfies
	\begin{align*}
	  \int_t^T\big\< \mu_s, \big| L^{\alpha}(s, \cdot, \mu_s)\big| \> \,ds
		~=&~
		\E\Big[\int_t^T \sum_{k\in K_s} \big|L^{\alpha}(s, X^{k}_s, \mu_s)  \big| \,ds \Big]\\
		~\leq&~
		C_L \E\Big[\int_t^T\sum_{k\in K_s}\big(1+|X^{k}_s|^2 +  |\alpha(s,X^k_s,\mu_s)|^2\big) \,ds\Big].
	\end{align*}
	In addition, by the linear growth condition on $\alpha$ from Definition~\ref{def:admissible}, it holds 
	\begin{align*}
	  \E\Big[\int_t^T\sum_{k\in K_s} |\alpha(s,X^k_s,\mu_s)|^2 \,ds\Big] 
		~\leq&~
		2 C^2_\alpha \E\Big[\int_t^T\sum_{k\in K_s}\big(1+|X^{k}_s|^2\big) \,ds\Big].
	\end{align*}
	Thus we have that
	\begin{align*}
	  \int_t^T\big\<\mu_s, \big| L^{\alpha}(s, \cdot, \mu_s)\big| \> \,ds
		~\leq&~
		C_L \big(1 + 2 C^2_\alpha\big) \E\Big[\int_t^T\sum_{k\in K_s}\big(1+|X^{k}_s|^2 \big) \,ds\Big],
	\end{align*}	
	and the latter is finite by Lemma~\ref{lemm:upperbound}. 
Similarly, we can deal with the term corresponding to the terminal cost as follows:
	\begin{align*}
	 \big\< \mu_T, \big| g(\cdot, \mu_T)\big| \>
	 =
		\E\Big[\sum_{k\in K_T} \big| g( X^{k}_T, \mu_T)  \big| \Big] 
		\leq
		C_g \E\Big[\sum_{k\in K_T}\big(1+|X^{k}_T|^2 \big)\Big].
	\end{align*}	
\end{proof}

\subsection{Dynamic Programming Principle}

	We can now establish the dynamic programming principle for the control problem introduced above. 
	It relies on the flow property satisfied by the flow of marginal measures
as established in the preliminary lemma below.

\begin{lemma}\label{lemma flow property}
		Let Assumption \ref{A.1} hold. 
		Let also $t\in [0,T],$ $m\in\Mc_2(\R^d)$ and $\alpha\in\Ac.$
		Then the flow of marginal measures $(\mu^{t, m,\alpha}_s)_{t\le s\le T}$ satisfies the flow property in the sense that
	\begin{align*}
		\mu^{t,m,\alpha}_\theta
		~=&~
		\mu^{s, \mu^{t,m,\alpha}_s,\alpha}_\theta,\quad \quad\text{for all } t \le s\le \theta \le T.
	\end{align*}
\end{lemma}

\begin{proof}
    Pick any $\xi \in \Xi_t^m$ so that $\mu^{t,m, \alpha} = \mu^{t,\xi, \alpha}.$ 
	By uniqueness of the solution to the McKean-Vlasov branching diffusion SDE~\eqref{eq:SDE_McKean-VlasovB}, we have the following cocycle property:
	\begin{align*}
		Z_\theta^{t,\xi,\alpha}
		~=&~
		Z_\theta^{s, Z^{t,\xi,\alpha}_s,\alpha},
		\quad\text{for all } t\le s\le \theta\le T.
	\end{align*}
	We deduce that
	\begin{align*}
		\mu^{t,m,\alpha}_\theta
		~=&~
		\mu^{t,\xi,\alpha}_\theta
		~=~
		\mu^{s, Z^{t,\xi,\alpha}_s,\alpha}_\theta
		~=~
		\mu^{s, \mu^{t,m,\alpha}_s,\alpha}_\theta.
	\end{align*}
\end{proof}

\begin{theorem}\label{DPP}
	Let Assumptions \ref{A.1} and \ref{A.3} hold.
	Let $v$ be the value function defined by~\eqref{eq:value function}.
	Then it holds for all $(t, m) \in [0,T] \x \Mc_2(\R^d),$ $s \in [t,T],$
	\begin{equation}\label{eq:dpp}
		v(t, m)
		~=~
		\inf_{\alpha\in\Ac}\bigg\{
			\int_{t}^{s} \<\mu^{t,m,\alpha}_\theta, L^{\alpha}(\theta,\cdot,\mu^{t,m,\alpha}_\theta)\> d\theta
		+
		v(s, \mu^{t,m,\alpha}_s)
		\bigg\}.
	\end{equation}
\end{theorem}

\begin{proof}
	We follow the pipeline of the proof for a deterministic dynamic programming principle.  We observe first that 
	\begin{align}
		J(t,m,\alpha) & ~=~ \int_{t}^{T} \<\mu^{t,m,\alpha}_\theta, L^{\alpha}(\theta,\cdot,\mu^{t,m,\alpha}_\theta)\> d\theta
		+
		\<\mu^{t,m,\alpha}_T, g(\cdot, \mu^{t,m,\alpha}_T)\> \nonumber\\
		 & ~ = ~	\int_{t}^{s} \<\mu^{t,m,\alpha}_\theta, L^{\alpha}(\theta,\cdot,\mu^{t,m,\alpha}_\theta)\> d\theta + J(s,\mu^{t,m,\alpha}_s, \alpha),\label{eq:conditionning}
	\end{align}      
	where we used the flow property $\mu^{t,m,\alpha}_\theta = \mu^{s,\mu^{t,m,\alpha}_s, \alpha}_\theta$ for $\theta\in[s,T]$ from Lemma \ref{lemma flow property} in the second equality.
	It follows that
	\begin{equation*}
	J(t,m,\alpha) ~ \geq ~
		\int_{t}^{s} \<\mu^{t,m,\alpha}_\theta, L^{\alpha}(\theta,\cdot,\mu^{t,m,\alpha}_\theta)\> d\theta
		+
		v(s, \mu^{t,m,\alpha}_s).
	\end{equation*}
	Taking the infimum over all controls $\alpha\in\Ac,$ we deduce the first inequality
	\begin{align*}
		v(t,m)
		~\geq~
		\inf_{\alpha\in\Ac}\bigg\{
			\int_{t}^{s} \<\mu^{t,m,\alpha}_\theta, L^{\alpha}(\theta,\cdot,\mu^{t,m,\alpha}_\theta)\> d\theta
		+
		v(s, \mu^{t,m,\alpha}_s)
		\bigg\}.
	\end{align*}
	For the reverse inequality, the idea is to concatenate two arbitrary controls, one before and one after the intermediate time $s.$ Namely,  we consider the control $\gamma(\theta,\cdot) := \1_{\theta\le s}\alpha(\theta,\cdot) + \1_{\theta>s}\beta(\theta,\cdot)$ with $\alpha, \beta\in\Ac$.  It is clear that $\gamma\in\Ac$ and it follows from~\eqref{eq:conditionning} and the definition of $\gamma$ that
	\begin{align*}
		v(t, m)
		~\leq ~
		J(t,m,\gamma)
		& ~=~
		\int_{t}^{s} \<\mu^{t,m,\gamma}_\theta, L^{\gamma}(\theta,\cdot,\mu^{t,m,\gamma}_\theta ) \> d\theta
		+
		J(s,\mu^{t,m,\gamma}_s, \gamma) \\
		& ~=~
		\int_{t}^{s} \<\mu^{t,m,\alpha}_\theta, L^{\alpha}(\theta,\cdot,\mu^{t,m,\alpha}_\theta)\> d\theta
		+
		J(s,\mu^{t,m,\alpha}_s, \beta).
	\end{align*}
    Taking the infimum over $\beta\in\Ac,$ we deduce that
	\begin{equation*}
		v(t,m)
		~\leq~
		\int_{t}^{s} \<\mu^{t,m,\alpha}_\theta, L^{\alpha}(\theta,\cdot,\mu^{t,m,\alpha}_\theta)\> d\theta
		+
		v(s, \mu^{t,m,\alpha}_s).
	\end{equation*}
	The reverse inequality then follows by taking the inifimum over $\alpha\in\Ac.$
\end{proof}

\subsection{Differentiation and It\^o's formula}

	We now define the notion of linear derivative for functionals defined on the space of finite measures,
	and then recall from Cao, Ren and Tan~\cite{CaoRenTan} It\^o's formula for the flow of marginal measures induced by McKean-Vlasov branching diffusion.

\begin{definition}\label{definition linear functional derivative}
	$\mathrm{(i)}$ For a function $F:\Mc_2(\R^d)\longrightarrow\R$,  the linear derivative, if it exists, is a continuous mapping $\delta_m F: \R^d\x\Mc_2(\R^d)\longrightarrow\R$ with at most quadratic growth in $x$ 
	satisfying
	\begin{align*}
		F(m) - F(m')
		~=~
		\int_0^1\int_{\R^d}\delta_m F(x,\lambda m+(1-\lambda) m')(m - m')(dx) d\lambda,
	\end{align*}
	for all $m,m'\in\Mc_2(\R^d).$
	
	\noindent $\mathrm{(ii)}$ 
	For a function $F: [0, T) \x \Mc_2(\R^d)\longrightarrow\R$, 
	we say that $F$ is of class $C^{1,2}([0,T) \x \Mc_2(\R^d))$ if the partial derivatives $\partial_t F$, $\delta_m F$, $\nabla_x \delta_m F$, $\nabla^2_x \delta_m F$ exist and are continuous, satisfying the following growth condition: there exists $C >0$ such that 
	\begin{align*}
	 \big|\delta_m F(s,x,m) \big| \vee\big|\nabla^2_x\delta_m F(s,x,m) \big| 
		\leq
		C
		\big(1+| x |^2 \big), 
		\quad 
		\big|\nabla_x\delta_m F(s,x,m) \big| 
		\leq
		C
		\big(1+| x | \big),
	\end{align*}	
	for all $(s,x,m)\in [0,T)\x \R^d\x\Mc_2(\R^d).$
\end{definition}

\begin{remark}
The mapping $\nabla_x \delta_m F$ is often called the intrinsic derivative in the literature and, unlike the standard case of probability distributions, the linear derivative $\delta_m F$ is uniquely defined here as elements of $\Mc_2(\R^d)$ can have different mass. 
\end{remark}

We are now in a position to state a variant of It\^o's formula, which together with the dynamic programming principle allow us to identify the Hamilton-Jacobi-Bellman equation satisfied by the value function.  Notice that the growth conditions imposed in Definition~\ref{definition linear functional derivative} are tailor-made to enforce integrability in the formula below.

\begin{proposition}\label{proposition Ito formula}
	Let $F\in C^{1,2}([0,T) \x \Mc_2(\R^d)).$ Let also $t\in[0,T),$ $m\in\Mc_2(\R^d),$ $\alpha\in\Ac.$
	Then it holds for all $s\in[t,T),$
	\begin{align}\label{eq:Ito_form}
		F(s, \mu^{t,m,\alpha}_s) 
		=  F(t, m) + 
		\int_t^s
		\big( \partial_t F(\theta, \mu^{t,m,\alpha}_\theta) 
		+
		\langle \mu^{t,m,\alpha}_\theta, \Gc^{\alpha}_\theta F(\theta,\cdot, \mu^{t,m,\alpha}_\theta) \rangle \big) 
		 \,d\theta,
	\end{align}
	where  for all $\theta\in[0,T)$ $x\in\R^d,$ $m\in\Mc_2(\R^d),$
\begin{multline*}
	\Gc^{\alpha}_\theta F(\theta,x,m)
	:=
	b^{\alpha}\big(\theta,x,m\big)\cdot \nabla_x \delta_m F(\theta,x,m)
	+
	\frac{1}{2}\Tr\big(\sigma^\alpha\sigma^{\alpha \top}\big(\theta,x,m\big) \nabla^2_x \delta_m F(\theta,x,m)\big)\\
	+
	\gamma^\alpha\big(\theta,x,m\big)\sum_{\ell\geq0}(\ell - 1)p_\ell^\alpha\big(\theta,x,m\big)\delta_m F(\theta,x,m).
\end{multline*}	
\end{proposition}

\subsection{Hamilton-Jacobi-Bellman Equation}

Let us denote by $\As$ the collection of all Lipschitz continuous maps $\as:\R^d\to A.$ Notice that admissible controls in $\Ac$ correspond to a subclass of mappings from $[0,T]\x\Mc_2(\R^d)$ to $\As.$

We start with a classical verification theorem which states that a smooth solution to the Hamilton-Jacobi-Bellman equation, if it exists, is unique and coincides with the value function. It also identifies an optimal control in a feedback form under stringent assumptions. 

\begin{theorem}\label{theorem varification}
	Let Assumptions \ref{A.1} and \ref{A.3} hold.  We assume that there exists a smooth function $u\in C^{1,2}([0,T)\x\Mc_2(\R^d))\cap C([0,T]\x\Mc_2(\R^d))$ satisfying
	\begin{equation}\label{eq:Hamilton-Jacobi-Bellman}
		\begin{cases}
			\partial_t u(t,m) + \inf_{\as\in\As}\big\{
				\< m, \Gc^{\as}_t u(t,\cdot,m) + L^{\as}(t,\cdot,m )\>
			\big\}
			~=~
			0, & \text{on } [0,T)\x \Mc_2(\R^d),\\
			u(T,m)
			~=~
			\<m, g(\cdot, m) \>,  & \text{on } \Mc_2(\R^d).
		\end{cases}
	\end{equation}
	Assume further that the infimum is attained by an admissible control in the sense that there exists $\hat\alpha\in \Ac$ such that
	\begin{equation}\label{eq:optimal}
	 \hat{\alpha}(t,\cdot,m) = \argmin_{\as\in\As}\big\{
				\<m, \Gc^{\as}_tu(t,\cdot,m) + L^{\as}(t,\cdot,m )\>
			\big\}.
	\end{equation}
	Then $u(t,m) = v(t,m) = J(t, m, \hat\alpha)$ for all $t\in[0,T],\,m \in \Mc_2(\R^d)$.
\end{theorem}

\begin{proof}
     Denote $T_\varepsilon := T-\varepsilon$ for $\varepsilon>0.$
     	We start by applying It\^o's formula~\eqref{eq:Ito_form} to $u$ from $(t,m)$ to $(T_\varepsilon,\mu^{t,m,\alpha}_{T_\varepsilon})$ with an arbitrary control $\alpha\in\Ac$ as follows:
	\begin{equation*}
		u(t,m)
		~=~
		u(T_\varepsilon, \mu^{t,m,\alpha}_{T_\varepsilon})
		-
		\int_t^{T_\varepsilon} \Big( \partial_{t} u(s,\mu^{t,m,\alpha}_s) +
			\big<\mu^{t,m,\alpha}_s, \Gc^{\alpha}_s u(s,\cdot, \mu^{t,m,\alpha}_s) \big\> 
			\Big) 
		\,ds.
	\end{equation*}
	Since $u$ is a continuous solution to the Hamilton-Jacobi-Bellman equation \eqref{eq:Hamilton-Jacobi-Bellman}, it follows that
	\begin{align*}
		u(t,m)
		~\leq&~
		u(T_\varepsilon, \mu^{t,m,\alpha}_{T_\varepsilon})
		+
		\int_t^{T_\varepsilon}
		\<\mu^{t,m,\alpha}_s, L^{\alpha}(s,\cdot,\mu^{t,m,\alpha}_s)\>
		\,ds 
		\xrightarrow[\varepsilon\to 0]{}
		J(t,m,\alpha).
	\end{align*}	
	By arbitrariness of $\alpha$,  we deduce the first inequality
	\begin{align}\label{eq:u leq v}
		u(t,m)
		\leq
		\inf_{\alpha \in \Ac} J(t, m, \alpha) = v(t,m).
	\end{align}
	For the reverse inequality, we repeat the same computation with the control $\hat\alpha$ defined by~\eqref{eq:optimal}, which attains the minimum in the Hamilton-Jacobi-Bellman equation. 
	We obtain
	\begin{align*}
		u(t,m)
		~=~
		u(T_\varepsilon, \mu^{t,m,\hat\alpha}_{T_\varepsilon})
		+
		\int_t^{T_\varepsilon}
		\big\<\mu^{t,m,\hat\alpha}_s, L^{\hat\alpha}(s,\cdot,\mu^{t,m,\hat\alpha}_s)\big\>
		\,ds
		\xrightarrow[\varepsilon\to 0]{}
		J(t,m,\hat{\alpha}).
	\end{align*}
	Together with \eqref{eq:u leq v}, we conclude that
	\begin{align*}
		v(t,m)
		~\geq~
		u(t,m)
		~=~
		J(t,m,\hat\alpha)
		~\geq~
		v(t,m),
	\end{align*}
	and thus equality holds.
\end{proof}

Next we prove a converse to the verification theorem establishing that, if the value function is smooth, then it necessarily satisfies the Hamilton-Jacobi-Bellman equation.

\begin{proposition}\label{Hamilton-Jacobi-Bellmanprop}
	Let Assumptions \ref{A.1} and \ref{A.3} hold. 
	Suppose in addition that the value function $v$ is of class $C^{1,2}([0,T)\x\Mc_2(\R^d))$. 
	Then it satisfies for all $t\in[0,T),$ $m\in\Mc_2(\R^d),$
	\begin{equation*}
			\partial_t v(t,m) + \inf_{\as\in\As}\big\{
				\<m, \Gc^{\as}_tv(t,\cdot,m) + L^{\as}(t,\cdot,m )\>
			\big\}
			~=~
			0.
	\end{equation*}
\end{proposition}

\begin{proof}
	Observe first that there is a natural embedding of $\As$ into $\Ac$ by letting $\alpha(t,x,m)=\as(x).$
	Then we can consider an arbitrary control $\as\in\As$ and a small time step $h>0.$ 
	By the dynamic programming principle~\eqref{eq:dpp}, we have 
	\begin{align}\label{ineq:1st inequality Hamilton-Jacobi-Bellman DPP}
		v(t,m)
		~\leq~
		\int_{t}^{t+h} \<\mu^{t,m,\as}_s,L^{\as}(s,\cdot,\mu^{t,m,\as}_s )\> \,ds
		+
		v(t+h, \mu^{t,m,\as}_{t+h}).
	\end{align}
	In addition, it follows from It\^o's formula~\eqref{eq:Ito_form} that
	\begin{align}\label{eq:1st inequality Hamilton-Jacobi-Bellman It\^o}
		v(t+h, \mu^{t,m,\as}_{t+h})
		~=&~
		v(t,m)
		+
		\int_t^{t+h} \Big(\partial_t v(s, \mu^{t,m,\as}_s)
		+
		\big\<
			\mu^{t,m,\as}_s, \Gc^{\as}_s v(s,\cdot, \mu^{t,m,\as}_s)
		\big\>
		\Big)
		 \,ds.
	\end{align}
	Combining \eqref{ineq:1st inequality Hamilton-Jacobi-Bellman DPP} and \eqref{eq:1st inequality Hamilton-Jacobi-Bellman It\^o}, we obtain
	\begin{align*}
		0
		~\leq~		
		\int_t^{t+h} \Big(\partial_t v(s, \mu^{t,m,\as}_s)
		+
		\big\<
			\mu^{t,m,\as}_s ,\Gc^{\as}_s v(s,\cdot, \mu^{t,m,\as}_s) +L^{\as}(s,\cdot,\mu^{t,m,\as}_s )
		\big\>\Big)
		 \,ds.
	\end{align*}
	Dividing by $h>0$ and letting $h\to 0,$ we deduce by continuity of the integrand that 
		\begin{equation*}\label{eq:continuity}
		0
		~\leq~
		\partial_t v(t,m)
		+
			\big\<
				m, \Gc^{\as}_t v(t,\cdot,m) + L^{\as}(t,\cdot,m)
			\big\>.
	\end{equation*}
	Notice that the continuity of the integrand comes from the continuity of $s\mapsto \mu^{t,m,\as}_s$ established in Lemma~\ref{lem:flow-continuity} and by dominated convergence theorem using the estimates 
	\begin{equation*}
	 \left|L^{\as}(s,x,m)\right| \le C (1 + |x|^2), \qquad \left|\Gc^{\as}_s v(s,x,m) \right| \le C (1 + |x|^2).
	\end{equation*}
Taking the infimum over all controls $\as\in\As,$ we obtain the first inequality
	\begin{equation*}\label{ineq:Hamilton-Jacobi-Bellman 1st inequality}
		0
		~\leq~
		\partial_t v(t,m)
		+
		\inf_{ \as\in\As}\big\{
			\big\<
				m, \Gc^{\as}_t v(t,\cdot,m) + L^{\as}(t,\cdot,m)
			\big\>
			\big\}.
	\end{equation*}
	For the reverse inequality, we assume by contradiction that there exists $(t,m)\in [0,T)\x\Mc_2(\R^d)$ and $\as\in\As$ such that  
	\begin{equation*}
		\partial_t v(t,m)
		+
			\big\<
				m, \Gc^{\as}_t v(t,\cdot,m) + L^{\as}(t,\cdot,m)
			\big\>
			~<~ 0.
	\end{equation*}	
	By continuity, there exists $h>0$ such that for all $s\in[t,t+h],$
		\begin{equation*}
		\partial_t v(s,\mu^{t,m,\as}_s)
		+
			\big\<
				\mu^{t,m,\as}_s, \Gc^{\as}_s v(s,\cdot, \mu^{t,m,\as}_s) + L^{\as}(s,\cdot,\mu^{t,m,\as}_s)
			\big\>
			~<~ 0.
	\end{equation*}
	Applying It\^o's formula~\eqref{eq:Ito_form}, we deduce that 	
	\begin{align*}
		v(t,m)  
		~=&~
		v(t+h, \mu^{t,m,\as}_{t+h})
		-		
		\int_t^{t+h} \Big(\partial_t v(s, \mu^{t,m,\as}_s)
		+
		\big\<
			\mu^{t,m,\as}_s, \Gc^{\as}_s v(s,\cdot, \mu^{t,m,\as}_s)
		\big\>
		\Big)
		 \,ds \\
		~>&~
		v(t+h, \mu^{t,m,\as}_{t+h})
		+		
		\int_t^{t+h} 
		\big\<
			\mu^{t,m,\as}_s,L^{\as}(s,\cdot,\mu^{t,m,\as}_s )
		\big\>
		 \,ds.
	\end{align*}
	This contradicts the dynamic programming principle in Theorem~\ref{DPP}.
\end{proof}

\section{Linear-Quadratic Example}
\label{sec: LQ}

	In this section, we provide a simple example of control problem in a linear quadratic setting.
	We give both an explicit solution to the Hamilton-Jacobi-Bellman equation and the corresponding optimal control by using the verification argument in Theorem~\ref{theorem varification}.

	For simplicity, we consider the one-dimensional setting and we introduce the following notations:
\begin{align*}
	\bar m
	:=
	m(\R),~~
	m_1
	:=
	\int_{\R} x \,m(dx),~~
	m_2
	:=
	\int_{\R}x^2 \,m(dx).
\end{align*}
	 The drift coefficient is assumed to be linear: for all $s\in[0,T], x\in\R, m\in\Mc_2(\R), a\in \R,$
\begin{equation*}\label{eqs: LQ coefficients}
	b(s, x, m, a) = b_1(s)\, x + b_2(s)\, \bar{m} + b_3(s)\, a,
\end{equation*}	
	where $b_1, b_2, b_3\in C([0,T],\R).$ 
	All the other coefficients are assumed to be constant: the volatility $\sigma\in\R,$ the death rate $\gamma\ge 0$ and the progeny distribution $(p_\ell)_{\ell\in\N}\in [0,1]^{\N}$ such that $\sum_{\ell\in \N} {p_\ell} =1$ and $\sum_{\ell\geq0}\ell p_\ell<+\infty.$
We also denote the growth rate of the population by $\kappa:=\gamma\sum_{\ell\geq0}(\ell-1)p_\ell.$
The cost functions are assumed to be quadratic: for all $s\in[0,T], x\in\R, m\in\Mc_2(\R), a\in \R,$
\begin{align*}
	L(s,x,m,a)
	~:=&~
	L_1 (s)\, x^2
	+
	L_2(s)\, \bar m
	+
	L_3(s)\, m_1
	+
	L_4(s)\, a^2,
	\\
	g(x,m)
	~:=&~
	g_1\, x^2 + g_2\, \bar m + g_3\, m_1,
\end{align*}
where $L_1,L_2,L_3,L_4\in C\big([0,T], \R)$  with $L_4>0$ and  $g_1,g_2,g_3\in\R$.

Let us now discuss how this linear-quadratic example compares to the assumptions of the paper. 
Regarding Assumptions~\ref{A.1}, it suffices to check it for the drift $b$ as all the other parameters are constant. Since the functions $b_1, b_2, b_3$ are continuous and thus bounded on $[0,T]$, the Lipschitz continuity comes easily from the following computation: for all $m,\tilde{m}\in\Mc(\R),$
 \begin{equation*}
  \big|m(\R) - \tilde{m}(\R)\big|= \big|\<m,1\> - \<\tilde{m},1\>\big|\leq \sup_{\varphi\in\text{BL}_1(\R)} \big\{\<m,\varphi\>-\<\tilde{m},\varphi\> \big\} =   \mathbf{d}_{\mathrm{BL}} (m, \tilde{m}).
 \end{equation*}
 As  for the linear growth condition,  it holds only locally in $\Mc(\R)$ in the sense that  for all $M>0,$ there exists $C_b^M>0$ such that 
	\begin{align*}
		\big|b(s, x ,m, a)\big|
		~&\leq~
		C^M_b\big(1+|x|  + |a|\big),
	\end{align*}
	for all $(s, x, m, a) \in [0,T] \x\R\x \Mc(\R)  \x A$ such that $m(\R)\le M.$ This is actually sufficient for our analysis to hold as the mass of the marginal measure remains bounded in view of~\eqref{ineq:supK}. 
	Similarly, regarding Assumptions~\ref{A.3},  the quadratic growth condition holds only locally in $\Mc_2(\R)$ in the sense that for all $M>0,$ there exists $C_L^M,C_g^M>0$ such that 
	\begin{align*}
	\big|L(s, x ,m) \big|
	 \le 
	 C^M_L
		\big(1 + |x|^2 + |a|^2 \big), 
		\qquad 
		\big|g( x ,m) \big| 
		\leq
		C^M_g
		\big(1+| x |^2 \big),
	\end{align*}
	for all $(s,x, m) \in [0,T]\x\R\x\Mc_2(\R)$ such that $\<m,1+|\cdot|^2\>\le M,$ which is also sufficient in view of Lemma~\ref{lemm:upperbound}. 

\begin{proposition}
	The value function of the linear quadratic problem introduced above is given by
	\begin{align*}
		v(t,m) 
		=
		\Lambda(t)m_2 +  \Gamma(t)\bar m m_1 +  \Gamma_1(t) \bar m  + \Gamma_2(t) \bar m ^2 + \Gamma_3(t)\bar m^3,
	\end{align*}
	and the corresponding optimal control by 
	\begin{align*}
		\hat{\alpha}(t, x, m) 
		= -\frac{b_3(t)}{2L_4(t)} \nabla_x \delta_m v(t,x,m)
		= -\frac{b_3(t)}{2L_4(t)} \big(2 \Lambda(t)\, x + \Gamma(t)\, \bar m\big),
	\end{align*}
	with $\Lambda$, $\Gamma,$ $\Gamma_i \in C^1([0,T], \R)$ for $i=1,2,3,$ being the (unique) solution to the following system of ordinary differential equations:
		\begin{equation}\label{ode system}
		\left\{ \begin{array}{rl}
					\Lambda'(t) -\frac{b_3(t)^2}{L_4(t)} \Lambda(t)^2 + \big(2 b_1(t) +\kappa\big)\Lambda(t) + L_1(t)  = 0,&  \Lambda(T) = g_1,\\ 
			\Gamma'(t) + \big(b_1(t) + 2\kappa - \frac{b_3(t)^2}{L_4(t)}\Lambda(t) \big)\Gamma(t) + 2b_2(t) \Lambda(t) + L_3(t) =0, & \Gamma(T) = g_3, \\
			\Gamma'_1(t) + \kappa\Gamma_1(t) + \sigma^2\Lambda(t) = 0,  &\Gamma_1(T) = 0, \\ 
			\Gamma'_2(t) + 2 \kappa \Gamma_2(t) + L_2(t)  = 0, &   \Gamma_2(T) = g_2, \\ 
			\Gamma'_3(t) + 3 \kappa \Gamma_3(t) + b_2(t)\Gamma (t) - \frac{b_3(t)^2}{4 L_4(t)}\Gamma(t)^2   = 0, &  \Gamma_3(T) = 0. 
		\end{array}\right.
	\end{equation}		
\end{proposition}

\begin{proof}
The proof relies on the verification theorem established in  Theorem~\ref{theorem varification}. 
 Let us look for a solution to the Hamilton-Jacobi-Bellman equation $w:[0,T]\x\Mc_2(\R)\to\R$  of the form
	\begin{align*}
		w(t,m) 
		~:=~ 
		\Lambda(t)m_2 + \Gamma(t)\bar m m_1 + \Gamma_1(t) \bar m  + \Gamma_2(t) \bar m ^2 + \Gamma_3(t)\bar m^3.
	\end{align*} 
	with unknown functions $\Lambda, \Gamma, \Gamma_1,  \Gamma_2, \Gamma_3$ to identify.  First we observe by a straightforward computation that for all $(t,x,m)\in[0,T]\x\R\x\Mc_2(\R),$
	\begin{align*}
		\delta_m w(t,x,m)
		~=&~ 
		\Lambda(t) x^2 + \Gamma(t)\bar m x + \Gamma(t) m_1 + \Gamma_1(t) + 2\Gamma_2(t) \bar m  + 3\Gamma_3(t)\bar m^2,\\
		\nabla_x \delta_m w(t,x, m)
		~=&~
		2\Lambda(t)x + \Gamma(t)\bar m,\\
		\nabla^2_x \delta_m w(t,x,m)
		~=&~
		2\Lambda(t).
	\end{align*}
	In addition, we have for all $\as\in\As,$
	\begin{align*}
	\< m, L^\as\big(t,\cdot, m\big)\>
	~=~ 
	L_1(t)m_2 + L_2(t)\bar m^2 + L_3(t)\bar m m_1 +\int_{\R}L_4(t)\as(x)^2 m(dx).
	\end{align*}
	It follows that
	\begin{multline}\label{LQ Hamilton-Jacobi-Bellman Gc}
		\<m, \Gc^\as_t w(t,\cdot,m) + L^\as\big(t,\cdot, m\big)\>
		=
		2(b_1\Lambda)(t) m_2 
		+
		(2b_2\Lambda + b_1\Gamma)(t) \bar m m_1
		+
		(b_2\Gamma)(t)\bar m ^3
		+
		 \sigma^2 \Lambda(t) \bar m \\
		\begin{aligned}
		& +
		\kappa\big(\Lambda(t)m_2 + 2\Gamma(t)\bar m m_1 + \Gamma_1(t)\bar m + 2\Gamma_2(t)\bar m^2+ 3\Gamma_3(t)\bar m^3\big) \\
		& + L_1(t)m_2 + L_2(t)\bar m^2 + L_3(t)\bar m m_1 + H^{\as}(t,m),
		\end{aligned}
	\end{multline}
	where the dependence in $\as$ is contained in the last term
	\begin{align*}
		H^{\as}(t,m) 
		~:=&~
		\int_{\mathbb{R}} L_4(t) \as(x)^2 m(dx)
		+
		\int_{\R} b_3(t) \big(2 \Lambda(t)x + \Gamma(t)\bar m\big) \as(x) m(dx).
	\end{align*}
	Next we observe that minimizing $H^{\as}$ over $\as\in\As$ is equivalent to minimizing the integrand 
	\begin{align*}
		L_4(t) \as(x)^2 +b_3(t) \big(2 \Lambda(t)x + \Gamma(t)\bar m\big) \as(x).
	\end{align*}
	The minimum is attained at 
	\begin{align*}
		\hat{\alpha}(t, x, m) 
		~=~
		-\frac{b_3(t)}{2L_4(t)} \big(2 \Lambda(t)\, x + \Gamma(t)\, \bar m\big),
	\end{align*}	
	and we have
	\begin{align}\label{LQ H optimal}
		H^{\hat{\alpha}}(t,m) 
		~=&~
		-\frac{b^2_3(t)}{4L_4(t)}\int_{\R} \Big(2 \Lambda(t) x + \Gamma(t) \bar m \Big)^2 m(dx)\nonumber\\
		~=&~
		- \frac{1}{L_4(t)}(b_3\Lambda)^2(t) m_2 - \frac{1}{L_4(t)}(b_3^2\Lambda\Gamma)(t) \bar m m_1 -\frac{1}{4L_4(t)}(b_3\Gamma)^2(t) \bar m^3.
	\end{align}	
	Therefore, combining \eqref{LQ Hamilton-Jacobi-Bellman Gc} and \eqref{LQ H optimal}, it holds that 
	\begin{multline*}
	   \partial_t w(t,m)  + \inf_{\as\in\As}\big\{
				\<m, \Gc^{\as}_t w(t,\cdot,m) + L^{\as}(t,\cdot,m )\>
			\big\}\\		
	 \begin{aligned}
		= ~ & 
		\Big(\Lambda'(t) + 2(b_1\Lambda)(t) + \kappa \Lambda(t) + L_1(t) - \frac{1}{L_4(t)}(b_3\Lambda)^2(t)  \Big)m_2 \\
		&~+
		\Big(\Gamma'(t) + (2b_2\Lambda + b_1\Gamma)(t) + 2\kappa \Gamma(t)+ L_3(t) - \frac{1}{L_4(t)}(b_3^2\Lambda\Gamma)(t)\Big)\bar m m_1\\
		&~+
		\Big(\Gamma'_1(t) + \sigma^2\Lambda(t) + \kappa \Gamma_1(t)\Big)\bar m
		+
		\Big(\Gamma'_2(t) + 2\kappa \Gamma_2(t) + L_2(t)\Big)\bar m^2\\
		&~+
		\Big(\Gamma'_3(t) + (b_2\Gamma)(t) + 3\kappa\Gamma_3(t) - \frac{1}{4L_4(t)}(b_3\Gamma)^2(t)\Big)\bar m^3.
		\end{aligned}
	\end{multline*}
	Regarding the terminal condition, we have
	\begin{align*}
		w(T,m) ~= ~ \Lambda(T)m_2  + \Gamma(T)\bar m m_1 + \Gamma_1(T) \bar m  + \Gamma_2(T) \bar m ^2 + \Gamma_3(T)\bar m^3,
	\end{align*}	
 and
	\begin{align*}
		\< m, g(\cdot, m)\>
		~=~
		g_1 m_2 + g_2 \bar m^2 + g_3 \bar m m_1.
	\end{align*} 
	Thus if the functions $\Lambda, \Gamma, \Gamma_1,  \Gamma_2, \Gamma_3$ satisfy the system of ODEs \eqref{ode system},  then $w$ is solution to the Hamilton-Jacobi-Bellman equation~\eqref{eq:Hamilton-Jacobi-Bellman}.
    In addition, since $\hat{\alpha}(t,x,m)$ is linear w.r.t.\ $x$ and $m,$ it defines an admissible control --- up to relaxing the linear growth condition as discussed above for the drift coefficient. 
    We conclude by Theorem~\ref{theorem varification} that $w(t,m)=v(t,m)=J(t,m,\hat{\alpha}).$
\end{proof}

\section{Complements on McKean-Vlasov Branching Diffusion}
\label{sec:complement}

Recall that one can think of a McKean-Vlasov branching diffusion process controlled by a closed-loop control $\alpha$ as an uncontrolled process with coefficient $(b^\alpha,\sigma^\alpha,\gamma^\alpha,p^\alpha_\ell).$ 
The aim of this section is to provide additional results on (uncontrolled) McKean-Vlasov branching diffusion which are used in the paper. 

\subsection{Existence and Uniqueness}

 The class of (uncontrolled) McKean-Vlasov branching diffusion processes was introduced in~\cite{Claisse Kang Tan 2024} in the path-dependent setting with coefficients depending on the whole distribution of the process.  However it was assumed that the drift coefficient is bounded and we need to extend some of the results therein to handle a linear growth condition. 
 For the sake of consistency and clarity,  we restrict to the case where the coefficient functions depend on the state variable in a Markovian way, while the dependence on the mean-field term is general, with the interaction measure variable lying in $\Pc(E)$.
	
	In this section, by abuse of notation, the coefficients of the McKean-Vlasov branching diffusion process are given by
	\[
	\big( b, \sigma, \gamma, (p_{\ell})_{\ell \in \N} \big)
	:
	[0,T] \times \R^d \times \Pc_1(E)
	~\longrightarrow~
	\R^d \times \R^{d \times d} \times [0, \bar{\gamma}] \times [0,1]^{\N},
	\]
	where $\sum_{\ell\in\N}{p_\ell(\cdot)} =1.$
Let the operator $\Lc$ be given by, for all $(s, x ,\nu)\in[0,T]\x\R^d\x\Pc_1(E)$ and $\varphi \in C^2_b(\R^d,\R)$,
	\begin{equation*} 
		\Lc \varphi (s, x, \nu)
		~:=~
		\frac{1}{2}\Tr\big(\sigma\sigma^{\top} (s, x,\nu)\nabla^2_x \varphi(x)\big) + b(s,x,\nu)\cdot \nabla_x \varphi(x).
	\end{equation*}
	Then the McKean-Vlasov branching diffusion with initial state $\xi$ at time $t$ can be characterized as the solution to the following SDE: for all $f:=(f^k)_{k\in\K}\in C^2_b (\K\x\R^d, \R),$ $s\in[t,T],$
\begin{multline} \label{eq:SDE_McKean-VlasovB-bis}
	\< Z_s, f \>
	= 
	\<\xi ,f\>
	+
	\int_t^s \sum_{k\in K_\theta}\Lc f^k(\theta, \Xk_\theta, \nu_\theta) \,d\theta
	+
	\int_t^s \sum_{k\in K_\theta}
	\nabla_x f^k(\Xk_\theta)\sigma(\theta,\Xk_\theta,\nu_\theta)
	\,d\Wk_\theta \\		
	+\int_{(t,s]\x [0,\gammab] \x [0,1]}
	\sum_{k\in K_{\theta-}}\sum_{ \ell \geq 0}
	\Big(\sum_{i=1}^{ \ell} 
	f^{ki} -f^k\Big)
	(\Xk_\theta)\1_{[0,\gamma(\theta,\Xk_\theta,\nu_\theta)]\x I_{\ell}(\theta,\Xk_\theta,\nu_\theta)}(z)\,\Qk( d\theta,dz),
\end{multline}
with the McKean-Vlasov condition
\begin{equation}\label{eq:def_mu_t-bis}
 \nu_s = \Lc(Z_s).
\end{equation}
We can show under the following assumptions that there exists a unique solution to this SDE.

	\begin{assumption}\label{A.1-bis}
    	\noindent $\mathrm{(i)}$ The death rate $\gamma$ is bounded by the constant $\gammab>0$ and the first moment of the progeny distribution  $\sum_{ \ell \ge 0} \ell p_{\ell}$ is bounded by a constant $M_1>0.$
	
	\noindent $\mathrm{(ii)}$ The coefficient functions $b,$ $\sigma,$ $\gamma,$ $(p_\ell)_{\ell\in\N}$ are Lipschitz in $(x,m)$ in the sense that, for each $\varphi = b, \sigma, \gamma, p_\ell,$ there exists $C_\varphi>0$ such that 
	\begin{align*}
		\big| \varphi(s, x ,\nu) - \varphi(s, x',\nu') \big|   
		~&\leq~
		C_\varphi
		\big( |x-x'| + \Wc_1(\nu, \nu')\big), 
	\end{align*}
	for all $(s, x, x', \nu, \nu') \in [0,T] \x \R^d \x \R^d \x \Pc_1(E) \x \Pc_1(E).$ 
	In addition, the Lipschitz constants $(C_{p_\ell})_{\ell\in\N}$ of the progeny distribution satisfy $\sum_{\ell \in\N} \ell C_{p_\ell} < \infty.$
	
	\noindent $\mathrm{(iii)}$ The volatility $\sigma$ is bounded by a constant $\bar{\sigma}>0$ and the drift $b$ satisfies a linear growth condition in $x$ in the sense that there exists $C_b>0$ such that 
	\begin{align*}
		\big|b(s, x ,\nu)\big|
		~&\leq~
		C_b\big(1+|x|\big),
	\end{align*}
	for all $(s, x, \nu) \in [0,T] \x \R^d \x \Pc_1(E).$
	\end{assumption}

\begin{proposition}\label{proposition existence and uniqueness-bis}
	Let Assumption \ref{A.1-bis} hold. Let $t\in [0,T]$ and $\xi\in \Xi_t.$ 
	Then there exists a unique (up to indistinguishability) $E$-valued adapted càdlàg process 
    	\begin{equation*}
		Z^{t,\xi}_s = \sum_{k\in K^{t,\xi}_s}\delta_{(k,X^{k}_s)}, \quad s\in [t,T],
	\end{equation*}
	satistyfing \eqref{eq:SDE_McKean-VlasovB-bis} together with the McKean-Vlasov condition \eqref{eq:def_mu_t-bis},
	such that 
	\begin{equation}\label{ineq:supK-bis}
		\E\Big[\sup_{s\in[t,T]} \# K^{t,\xi}_s\Big]
		~\leq~
		\E\big[\<\xi, 1\>\big] e^{\bar\gamma M_1 (T-t)}.
	\end{equation}	
\end{proposition}

\begin{proof}
   It is a slight generalization of Theorem~1 in \cite{Claisse Kang Tan 2024}. The only difference is that we now allow the drift $b$ to have linear growth instead of being uniformly bounded.  The proof works by the same arguments, using a Banach fixed point theorem coming from a contraction property on short time horizon.  To this end, the crucial estimate is provided by Lemma~3 in \cite{Claisse Kang Tan 2024}.  This is exactly where we need to amend slightly the proof to account for the linear growth of the drift. 
   More precisely, using the notations from \cite{Claisse Kang Tan 2024}, we need to replace Equation~(19) in Step~2 by a localized version as follows: 
   recall that $\|\cdot\| = \|\cdot\|_\infty \wedge 1,$ we have   
    \begin{multline*} \label{eq:ineq_path}
    \1_{k\in K^{1\cap2}_t} \big\|\Xmuk_{t\wedge\cdot}-\Xnuk_{t\wedge\cdot}\big\|
    \le
    \1_{k\in K^{1\cap2}_t} \big\|\Xmuk_{\bar{S}^k\wedge\cdot}-\Xnuk_{\bar{S}^k\wedge\cdot}\big\|  \\
    + 
    \int_0^{t \wedge \tau^k}
                \1_{k\in K^{1\cap2}_s}\big| b^{1,k}_s - b^{2,k}_s \big|
            ds 
         +
           \sup_{0\le r\le t \wedge \tau^k} \Big|\int_0^r
                \1_{k\in K^{1\cap2}_s}\big(\sigma^{1,k}_s - \sigma^{2,k}_s\big)
            d\Wk_s\Big|,
    \end{multline*}
    which comes by observing that
        $$
	\big\|\Xmuk_{t\wedge\cdot}-\Xnuk_{t\wedge\cdot}\big\|
	=
	\big\|\Xmuk_{t \wedge \tau^k \wedge\cdot}-\Xnuk_{t \wedge \tau^k \wedge\cdot}\big\|,
	~~\mbox{with}~~
	\tau^k := \inf \big\{s \ge \bar{S}^k ~: \big\|\Xmuk_{s\wedge\cdot}-\Xnuk_{s\wedge\cdot}\big\| \ge 1 \big\}.
    $$
    Then the drift term can be treated as before
by means of the Lipschitz condition on $b$:
    \begin{align*}
        \E\Big[
          \int_0^{t \wedge \tau^k}
                \1_{k\in K^{1\cap2}_s}
                \big|
                    b^{1,k}_s - b^{2,k}_s
                \big|
            ds 
        \Big] 
       & \leq 
        C \E\Big[
         	\int_0^{t}
                \1_{k\in K^{1\cap2}_s} 
                \Big(
                    \big\|\Xmuk_{s\wedge\cdot}-\Xnuk_{s\wedge\cdot}\big\|
                    +
                    \Wc_1\big(\mu^1_s,\mu^2_s\big)
                \Big) ds
        \Big].
    \end{align*}
    The rest of the proof can be left completely unchanged.
\end{proof}

We next provide an a priori estimate on the second moment of the McKean-Vlasov branching diffusion process.

\begin{lemma}\label{lemm:upperbound-bis}
 Under the conditions of Proposition~\ref{proposition existence and uniqueness-bis}, if we assume further that $\E[\<\xi, |\cdot|^2\>] < \infty,$ then there exists $C>0$ such that 
	\begin{align*}
		\E\bigg[\sup_{s\in [t,T]} \sum_{k\in K^{t,\xi}_s} \big( 1 + |X^{k}_s|^2\big) \bigg] \le C \E\big[\<\xi, 1+|\cdot|^2\>\big].
	\end{align*}
\end{lemma}
     
\begin{proof}
    For ease of notation, we omit the superscripts $t,\xi$ in the proof.  
	Let $(\tau_n)_{n\in\N}$ be a sequence of stopping times defined by
	\begin{equation*}\label{eq:ST1}
		\tau_n
		:=
		\inf\big\{s \geq t : \exists\, k\in K_s ~~ |X^{k}_s|
		\geq n\big\}.
	\end{equation*}
	Applying It\^o's formula~\eqref{eq:SDE_McKean-VlasovB-bis} with $f^k(x):=|x|^2$ up to time $\tau_n$, we obtain that for any $s\in[t,T],$ 
	\begin{multline*}
		\sum_{k \in K_{s \wedge \tau_n}} | X_{s \wedge \tau_n}^{k} \big|^2
		~=~
		\< \xi, 1 + |\cdot|^2 \> 
		+
		\int_t^{s \wedge \tau_n} \sum_{k \in K_\theta} 2 {\Xk_\theta}^\top \sigma(\theta,\Xk_\theta,\nu_\theta) \,dW_\theta^k \\
		\begin{aligned}
		&+
		\int_t^{s \wedge \tau_n} \sum_{k \in K_\theta} \Big( 2  X_\theta^{k}  \cdot b (\theta,\Xk_\theta,\nu_\theta )  +  \Tr \big( \sigma \sigma^{\top} (\theta,\Xk_\theta,\nu_\theta) \big) \Big)  \,d\theta \\
		&+ 
		\int_{(t, s \wedge \tau_n] \times [0,\gammab] \times [0,1]}
		\sum_{k\in K_{\theta-}}\sum_{ \ell \geq 0}
		(\ell - 1)
		 |X_\theta^k|^2
		\1_{[0,\gamma(\theta,\Xk_\theta,\nu_\theta)]\x I_{\ell}(\theta,\Xk_\theta,\nu_\theta)}(z)\,\Qk( d\theta,dz).
    \end{aligned}
	\end{multline*}
	Next we consider the supremum over the interval $[t,T]$ and take expectation. Then we study the upper bounds of each term on the right hand side. By martingale property, we can apply the Burkholder-Davis-Gundy inequality to the second term: there exists a constant $C_1>0$ such that
	\begin{multline*}		
	\begin{aligned}
		\mathbb{E} \bigg[ \sup_{s \in [t, T]} \int_t^{s \wedge \tau_n} \sum_{k \in K_\theta} 2 {X_\theta^{k}}^\top \sigma (\theta,\Xk_\theta,\nu_\theta ) \,d\Wk_\theta\bigg]
		~\leq&~ 
		C_1 \|\sigma\|_{\infty}^2  \mathbb{E} \bigg[ \bigg( \int_t^{T\wedge \tau_n}  \sum_{k \in K_\theta} | X_\theta^{k} |^2  d\theta \bigg)^{1/2} \bigg]\\
		~\leq&~
		 C_1 \|\sigma\|_{\infty}^2 \E \bigg[ \1_{\#K_t > 0} + \int_t^{T\wedge \tau_n} \sum_{k \in K_\theta}  | X_\theta^{k} |^2  \,d\theta\bigg],
	 \end{aligned}
	\end{multline*}
	where we applied the inequality $\sqrt{x}\leq 1 + x$ for $x >0$ in the last line. Now we deal with the third term associated with the infinitesimal generator.  
   In view of the linear growth assumptions on $b,$ it follows that
	\begin{multline*}
		\mathbb{E} \bigg[ \sup_{s \in [t, T]}
		\int_t^{s \wedge \tau_n} \sum_{k \in K_\theta} \Big( 2  X_\theta^{k}  \cdot b (\theta,\Xk_\theta,\nu_\theta )  +  \Tr \big( \sigma \sigma^{\top} (\theta,\Xk_\theta,\nu_\theta) \big) \Big)  \,d\theta \bigg]\\
		\begin{aligned}
		~\leq&~  
		\mathbb{E} \bigg[ \int_t^{T \wedge \tau_n} \sum_{k \in K_\theta} \bigg( | X_\theta^{k} |^2 + |b(\theta,\Xk_\theta,\nu_\theta)|^2 + \|\sigma\|_\infty^2 \bigg) \,d\theta  \bigg]\\
		 ~\leq&~ 
		 \big( 2C_b^2 + \|\sigma\|_\infty^2 + 1\big) \mathbb{E} \bigg[ \int_t^{T\wedge \tau_n}  \sum_{k \in K_\theta} \big(1 + | X_\theta^{k} |^2\big)  \,d\theta \bigg].
		 \end{aligned}
	\end{multline*}
	Lastly, the branching term can be treated as follows:
	\begin{multline*}
		\mathbb{E} \bigg[ \sup_{s \in [t, T]} \int_{(t, s \wedge \tau_n] \times [0,\gammab] \times [0,1]}
		\sum_{k\in K_{\theta-}}\sum_{ \ell \geq 0}
		(\ell - 1)
		|X_\theta^k|^2
		\1_{[0,\gamma(\theta,\Xk_\theta,\nu_\theta)]\x I_{\ell}(\theta,\Xk_\theta,\nu_\theta)}(z) \,\Qk( d\theta,dz) \bigg]\\
		\begin{aligned}
		~\leq&~ 
			\mathbb{E} \bigg[ \int_{(t, T \wedge \tau_n] \times [0,\gammab] \times [0,1]}
		\sum_{k\in K_{\theta-}}\sum_{ \ell \geq 1}
		(\ell - 1)
		|X_\theta^k|^2
		\1_{[0,\gamma(\theta,\Xk_\theta,\nu_\theta)]\x I_{\ell}(\theta,\Xk_\theta,\nu_\theta)}(z) \,\Qk( d\theta,dz) \bigg]\\
		~\leq&~ 
		\mathbb{E} \bigg[ \int_t^{T \wedge \tau_n} \sum_{k \in K_{\theta}} \gamma(\theta,\Xk_\theta,\nu_\theta) \sum_{\ell \geq 1} (\ell - 1) |X_\theta^k|^2  p_\ell (\theta,\Xk_\theta,\nu_\theta)  \,d\theta \bigg]\\
		~\leq&~
		\bar\gamma M_1 \mathbb{E} \bigg[ \int_t^{T \wedge \tau_n} \sum_{k \in K_\theta}  |X_\theta^k|^2 \,d\theta \bigg].
		\end{aligned}
	\end{multline*}
	By combining all the estimates above, we deduce that for some constant $C>0$ depending on $\bar{\gamma}, M_1, \|\sigma\|_{\infty}, C_b, T,$
	\begin{align*}
		\mathbb{E} \bigg[ \sup_{s \in [t, T]} \sum_{k \in K_{s \wedge \tau_n}} | X_{s \wedge \tau_n}^{k} |^2  \bigg]
		& \le
		\E\big[\<\xi, |\cdot|^2\>\big] + C \E \bigg[ \1_{\#K_t > 0} + \int_t^{T \wedge \tau_n} \#K_\theta \,d\theta +  \int_t^{T \wedge \tau_n} \sum_{k \in K_\theta}  |X_\theta^k|^2 \,d\theta \bigg] \\
		& \le
		\E\big[\<\xi, |\cdot|^2\>\big] + C \E\big[\<\xi, 1\>\big] + C \E \bigg[\int_t^{T} \sup_{s\in [t, \theta]} \sum_{k \in K_{s\wedge \tau_n}} | X_{s\wedge \tau_n}^{k} |^2  \,d\theta \bigg],
	\end{align*}
	by using
	\begin{equation*}
	 \E \bigg[ \1_{\#K_t > 0} + \int_t^{T \wedge \tau_n} \#K_\theta \,d\theta \bigg] \le (1 + T)\, \mathbb{E} \bigg[ \sup_{s\in [t, T]} \# K_s \bigg] \le  (1+T)\, \E\big[\<\xi, 1\>\big] e^{\gammab M_1 T}.
	\end{equation*} 
	Then we can apply Gr\"onwall Lemma to obtain
	\begin{align*}
		\mathbb{E} \bigg[ \sup_{s \in [t, T]} \sum_{k \in K_{s \wedge \tau_n}}  | X_{s\wedge \tau_n}^{k} |^2 \bigg]
		\leq 
		 C \E\big[\<\xi, 1+|\cdot|^2\>\big].
	\end{align*}
	The constant $C > 0$ being independent of $n,$
	the conclusion follows by sending $n$ to infinity. 
\end{proof}

\subsection{Stability in Distribution}

The main result of this section is a stability property for SDE~\eqref{eq:SDE_McKean-VlasovB-bis}.  It is a key step toward the proof of the invariance principle in Section~\ref{sec:marginal}. 
Note that a similar result can be found in~\cite[Proposition~4]{Claisse Kang Tan 2024}, where the authors established a property of pathwise stability under the assumption that the coefficients are uniformly convergent.  In the present setting, we show that we can relax this condition by considering stability in distribution.

\begin{proposition}\label{prop:stability} 
 Let $(\xi_n)_{n\in\N}$ be a sequence of initial state in $\Xi_0$ such that $\sup_{n\in\N}\E[\<\xi_n, 1 + |\cdot|^2\>]<+\infty.$ Let also $(b_n,\sigma_n,\gamma_n,p_\ell^n)_{n\ge 0}$ be a sequence of coefficients satisfying Assumption~\ref{A.1-bis} with identical constants $(\gammab,M_1,C_b,C_\sigma,C_{\gamma},C_{p_\ell},\bar{\sigma}).$ 
Assume further that 
\begin{itemize}
 \item[\rm(i)] there exists a $\N$-valued random variable $\bar{N}_0$ such that $\P(\<\xi_n,1\>\ge N) \le \P(\bar{N_0}\ge N)$ for all $N\ge 1, n\ge 0,$ 
 
 \item[\rm(ii)] there exist a probability mass function $(\bar{p}_\ell)_{\ell\ge 1}$ and a constant $C>0$ such that $\sum_{\ell\ge 1} {\ell\bar{p}_\ell} <+\infty$ and  $p^n_\ell(\cdot)\le C \bar{p}_\ell$ for all $\ell\ge 1, n\ge 0.$ 
\end{itemize}
If $(b_n,\sigma_n\sigma_n^\top,\gamma_n,p_\ell^n)$ converges pointwise to $(b,\sigma\sigma^\top,\gamma,p_\ell)$ and $\xi_n$ converges weakly to $\xi,$ then the sequence of solutions $Z^n$ to SDE~\eqref{eq:SDE_McKean-VlasovB-bis} with coefficient $(b_n,\sigma_n,\gamma_n,p^n_\ell)$ and initial state $\xi_n$ converges weakly to the solution $Z$ to SDE~\eqref{eq:SDE_McKean-VlasovB-bis} with coefficient $(b,\sigma,\gamma,p_\ell)$ and initial state $\xi$ in the sense that $\P\circ (Z^{n})^{-1}$ converges to $\P\circ Z^{-1}$ in $\Pc(D([0,T],E)).$ 
\end{proposition}

\begin{proof} 
 The proof proceeds in four steps. We show first that the sequence $\Pt_n := \P\circ (Z^n)^{-1}$ is tight in $\Pc(D([0,T],E)).$ Then we identify any limit point as a solution to an appropriate martingale problem. Finally we conclude by 
uniqueness to SDE~\eqref{eq:SDE_McKean-VlasovB-bis}. 
 Let us introduce the canonical space 
as $\tilde \Om := D([0,T],E)$ the set of all $E$-valued c\`adl\`ag paths on $[0,T]$.
	It is equipped with the canonical process $\tilde Z_t = \sum_{k \in \tilde K_t} \tilde X^k_t$ for $t \in [0,T]$ and the canonical filtration $\tilde \F = (\tilde \Fc_t)_{0 \le t \le T}$ generated by $\tilde Z$.

   \noindent\emph{Step 1.} Let us show first that the sequence $\Pt_n$ is tight in $\Pc(D([0,T],E)).$ It comes by similar arguments as~\cite[Proposition~3]{Claisse Kang Tan 2024}, which we reproduce here for the sake of completeness.  
We aim to apply Aldous' criterion of tightness, see, \eg, Billingsley~\cite[Theorem 16.10]{Billingsley 2013}. 
	The condition of uniform bound comes by Markov's inequality and~\eqref{ineq:supK-bis} as follows: 
		\begin{equation*}
		 \P \Big(\sup_{0\leq t\leq T} d_E(Z^n_t,e_0) \geq N\Big)\leq \frac{1}{N} {\E\Big[ \sup_{0\leq t\leq T}  \# K^n_t\Big]} \leq \frac{e^{\gammab M_1 T}}{N} \sup_{n\in\N} \E\big[\<\xi_n, 1\>\big]\xrightarrow[N\to +\infty]{} 0,
		\end{equation*}
		where we recall that $d_E$ is defined by~\eqref{eq:def_d_E} and $e_0$ is the null element of $E$ so that $d_E(Z^n_t,e_0)=\# K^n_t.$
		It remains to check the condition  of equicontinuity which writes as follows: for all $\varepsilon>0$, 
		\begin{equation} \label{eq:Aldous2}
			\lim_{\delta\rightarrow0}\sup_{n\in\N}\sup_{\tau \in\Tc}\P \big(d_E(Z^n_\tau, Z^n_{(\tau+\delta)\wedge T})>\varepsilon\big) = 0,
		\end{equation}
		where $\Tc$ is the collection of all $[0,T]$--valued stopping times on $(\Om, \Fc, \F).$
      Given $\delta >0,$ $n\in\N$ and $\tau\in\Tc,$ denote $\tau':= (\tau + \delta)\wedge T$ and observe that
	\begin{multline}\label{eq:aldous}
		d_E(Z^n_\tau,Z^n_{\tau'})
		\leq
		\sum_{k\in K^n_\tau}\big|X^{n,k}_{\tau'} - X^{n,k}_{\tau}\big|\wedge1\\
		 +
		\int_{(\tau,\tau']\x[0,\gammab] \x [0,1]}
		\sum_{k\in K^n_{s-}}
		\sum_{\ell \geq0}
		(\ell +1)\1_{I^n_{\ell} (s,X^{n,k}_s,\nu^n_s)\x[0,\gamma_n(s,X^{n,k}_s,\nu^n_s)]}(z)
		\,Q^k(ds,dz).
	\end{multline}
	The first term neglects potential death of particles, while the second term counts the total number of particles born or dead between $\tau$ and $\tau'.$ 
There is slight abuse of notation here as we need to extend the path of particles who have died before time $\tau'$ as solution to  SDE
	\begin{equation*}
	d X^{n,k}_s ~=~ b_n(s, X^{n,k}_s,\nu^n_s)\,ds+\sigma_n(s, X^{n,k}_s,\nu^n_s)\,d\Wk_s.
	\end{equation*}
	We start by dealing with the first term on the rhs of~\eqref{eq:aldous}.
    Notice that
    \begin{align*}
        &\E\Big[
            \sum_{k\in K^n_\tau}\big|
                X^{n,k}_{\tau'}-X^{n,k}_\tau
            \big|
        \Big]
        \leq
        \E\Big[
            \sum_{k\in K^n_\tau}
            \Big(\int_\tau^{\tau'}
                \big|b_n(s,X^{n,k}_s,\nu^n_s)\big|
            \,ds
        +
                \Big|
                    \int_\tau^{\tau'}
                        \sigma_n(s,X^{n,k}_s,\nu^n_s)
                    \,dW^k_s
                \Big|
            \Big)
        \Big].
    \end{align*}
    On the one hand, we have
    \begin{align*}
        \E\Big[
            \sum_{k\in K^n_\tau}\int_\tau^{\tau'}
                \big|b_n(s,X^{n,k}_s,\nu^n_s)\big|
            \,ds
        \Big]
        \leq
         C_b \E\Big[
            \sum_{k\in K^n_\tau}\int_\tau^{\tau'}
                \big(1 + \big|X_s^{n,k}\big|\big)
            \,ds
        \Big] 
         \leq
         C \delta\, \E\Big[
            \sum_{k\in K^n_\tau}
                \big(1 + \big|X_\tau^{n,k}\big|\big)
            \,ds
        \Big],
    \end{align*}
    where we used the classical estimate for diffusion $$\E\big[|X_s^{n,k}| \,|\, \Fc_\tau\big]\le C(1+|X_\tau^{n,k}|),\quad \text{ for all }\tau\le s\le T.$$
    On the other hand, it follows from Burkholder-Davis-Gundy inequality that
    \begin{align*}
        \E\Big[
            \sum_{k\in K^n_\tau}
                \Big|
                    \int_\tau^{\tau'}
                        \sigma_n(s,X^{n,k}_s,\nu^n_s)
                    dW^k_s
                \Big|
        \Big]
        & \leq
         C_1  \bar{\sigma} \sqrt{\delta}\, \E\big[\# K^n_\tau\big].
    \end{align*}
    As for the second term on the rhs of~\eqref{eq:aldous}, a direct computation yields
    \begin{multline}
        \E\bigg[
            \int_{(\tau,\tau']\x[0,\gammab] \x [0,1]}
                \sum_{k\in K^n_{s-}}
                    \sum_{\ell\geq0}
                        (\ell+1)\1_{I^n_\ell(s,X^{n,k}_s,\nu^n_s)\x[0,\gamma_n(s,X^{n,k}_s,\nu^n_s)]}(z)
            \,Q^k(ds,dz)
        \bigg]\nonumber\\
        \begin{aligned}
        &=
        \E\bigg[
            \int_\tau^{\tau'}
                \sum_{k\in K^n_s}\gamma_n(s,X^{n,k}_s,\nu^n_s)\sum_{\ell\geq0}(\ell+1)p^n_\ell(s,X^{n,k}_s,\nu^n_s)
            \,ds
        \bigg]\nonumber\\
        & \leq  
        \gammab (M_1+1) \delta\,\E\Big[
            \sup_{\tau\leq t\leq \tau'}\# K^n_t
        \Big].
        \end{aligned}
    \end{multline}
   It follows from the above estimates and Lemma~\ref{lemm:upperbound-bis} that
    \begin{align*}
        \E\big[
            d_E(Z^n_\tau,Z^n_{(\tau+\delta)\wedge T})
        \big]
        \leq 
        C\big(\delta + \sqrt{\delta}\big)\, \sup_{n\in\N} \E\big[\<\xi_n, 1 + |\cdot|^2\>\big].
    \end{align*}
    We deduce that the second condition \eqref{eq:Aldous2} holds by Markov's inequality.
    
   \noindent\emph{Step 2.}  Next consider an extracted subsequence, still denoted by $\tilde{\P}_n,$ converging to a limit point   $\tilde{\P}$ in $\Pc(D([0,T],E)).$ We aim to show that the marginal $\tilde{\P}_n\circ\tilde{Z}_t^{-1}$ also converges toward  $\tilde{\P}\circ\tilde{Z}_t^{-1}$ in $\Pc_1(E)$ (under the Wasserstein distance $\Wc_1$) for almost all $t\in[0,T].$  
   First we observe that the convergence occurs  in $\Pc(E)$ (under the weak convergence topology) outside of a countable subset of $[0,T]$ by property of the Skorokhod topology, see, \eg, Jacod and Shiryaev~\cite[Proposition~VI.3.14]{Jacod 2003}. 
Then it remains to prove that the sequence $(\# K_t^n)_{n\in\N}$ is uniformly integrable. This follows from the additional assumptions (i) and (ii) of Proposition~\ref{prop:stability}  which yield
\begin{equation*}
 \P\big(\# K_t^n \ge N\big) \le \P\big( \bar{N}_t \ge N\big), \quad \text{for all } N\ge 1, n\ge 0,
\end{equation*}
where $(\bar{N}_t)_{t\ge 0}$ is a branching process with branching rate $C\gammab$ and progeny distribution $(\bar{p}_\ell)_{\ell\ge 1}.$ Indeed, denote $N_t^n := \# K_t^n$ and observe that, by It\^o's formula,
    \begin{multline*}
        \1_{N^n_t\ge N} 
         =  \1_{N^n_0\ge N} +
         \int_{(0,t]\x[0,\gammab] \x [0,1]}
                \sum_{k\in K^n_{s-}}
                    \sum_{\ell\geq0}
                        \Big( \1_{N^n_{s-} + \ell-1 \ge N} -  \1_{N^n_{s-}\ge N} \Big) \\
                       \x \1_{I^n_\ell(s,X^{n,k}_s,\nu^n_s)\x[0,\gamma_n(s,X^{n,k}_s,\nu^n_s)]}(z)
            \,Q^k(ds,dz)
    \end{multline*}
    Taking expectation, we deduce that
        \begin{multline*}
        \P\big(N^n_t\ge N\big) \\
        \begin{aligned}
         & \le  \P\big(N^n_0\ge N\big) 
          + 
         \E \Big[ \int_0^t \sum_{k \in K_{s}} \gamma(s,X^{n,k}_s,\nu^n_s) \sum_{\ell \geq 1}
         p_\ell (s,X^{n,k}_s,\nu^n_s)
         \Big( \1_{N^n_{s-} + \ell-1 \ge N} -  \1_{N^n_{s-}\ge N} \Big) 
           \,ds
            \Big] \\
         & \le  \P\big(N^n_0\ge N\big) 
          + 
         C\bar{\gamma} \int_0^t  \sum_{\ell \geq 1}
          \bar{p}_\ell 
         \Big(  \P\big(N^n_s + \ell-1 \ge N\big) - \P\big(N^n_s\ge N\big) \Big) 
         \,ds.
         \end{aligned}
    \end{multline*}
    The conclusion now follows from the fact that  $\P(N^n_0\ge N) \le \P(\bar{N}_0\ge N)$ and
      \begin{equation*}
        \P\big(\bar{N}_t\ge N\big)
	    = 
	    \P\big(\bar{N}_0\ge N\big) 
          + 
         C\bar{\gamma} \int_0^t  \sum_{\ell \geq 1}
          \bar{p}_\ell 
         \Big(  \P\big(\bar{N}_s + \ell-1 \ge N\big) - \P\big(\bar{N}_s\ge N\big) \Big) 
         \,ds.
    \end{equation*}

   \noindent\emph{Step 3.} Next we identify any limit point of the sequence $(\Pt_n)$ as a solution to an appropriate martingale problem.  
	Given $\Phi\in C^2_b(\R,\R)$ and $\varphi = (\varphi^k)_{k\in \K} \in C^2_b(\K\x\R^d,\R),$ let us define on the canonical space the process
	\begin{align*}
		\tilde\Mc^{\Phi_\varphi}_t
		:=
		\Phi_{\varphi}(\Zt_t)-\int_0^t \Hc \Phi_{\varphi}(s, \Zt_s, \nu_s) \,ds, \quad t\in [0,T],
	\end{align*}			
	where $\nu_s = \Pt\circ\Zt^{-1}_s$ and
	\begin{multline*}
		\Hc \Phi_{\varphi}(s, \Zt_s, \nu_s)
		:=
		\frac{1}{2}\Phi_{\varphi}''(\Zt_s) \sum_{k\in \Kt_s} \big|D\varphi^k(\Xt^k_s) \sigma(s, \Xt^k_s, \nu_s)\big|^2
		+
		\Phi_{\varphi}'(\Zt_s) \sum_{k\in \Kt_s} \Lc \varphi^k(s, \Xt^k_s, \nu_s)\\
		 +
		\sum_{k\in \Kt_s}\gamma(s, \Xt^k_s,\nu_s)
		\Big(
			\sum_{\ell \ge 0} \Phi_{\varphi}\Big( \Zt_{s-} -\delta_{(k, \Xt^k_{s})}+\sum_{j = 1}^{\ell} \delta_{(kj, \Xt^{k}_{s} )} \Big) p_{\ell} (s, \Xt^k_s,\nu_s)-\Phi_{\varphi}(\Zt_{s-})
		\Big).
	\end{multline*}	
%
%
		We can then show by classical arguments that any limit point $\Pt$ of the sequence $(\Pt_n)$ solves the following martingale problem:
		\begin{itemize}
			\item[(i)] $\Pt \circ \tilde{Z}_0^{-1} = \Lc(\xi)$,
			
		
			\item[(ii)] $\tilde\Mc^{\Phi_\varphi}$ is a $(\Pt, \Ft)$--martingale for all $\Phi\in C^2_b(\R,\R)$ and $\varphi = (\varphi^k)_{k\in \K} \in C^2_b(\K\x\R^d,\R).$
		\end{itemize}
		Indeed, the first point comes immediately by passing to the limit in the relation $\Pt_n \circ \tilde{Z}_0^{-1} = \Lc(\xi_n)$ since the projection $\tilde{Z}_0$ is continuous for the Skorokhod topology, see, \eg, Billingsley~\cite[Theorem~12.5]{Billingsley 2013}. 
		Regarding the second point,  it comes by observing first that 
		 the process $\tilde\Mc^{\Phi_\varphi, n}$ defined as
	\begin{align}\label{eq:martingale}
		\tilde\Mc^{\Phi_\varphi, n}_t
		:=
		\Phi_{\varphi}(\Zt_t)-\int_0^t \Hc^n \Phi_{\varphi}(s, \Zt_s, \nu^n_s) ds, \quad t\in [0,T],
	\end{align}
	is a $(\Pt_n, \Ft)$--martingale by It\^o's formula, where $\Hc^n$ is defined as $\Hc$ with coefficients $(b_n,\sigma_n,\gamma_n,p_\ell^n)$ and $\nu^n_s = \Pt_n\circ\tilde{Z}_s^{-1} .$ Observe that $\nu^n_s$ converges to $\nu_s$ in $\Pc_1(E)$ for almost all $s\in[0,T]$ in view of Step~2 above.  This yields the convergence of $\Hc^n \Phi_{\varphi}(s, e, \nu^n_s)$ toward $\Hc \Phi_{\varphi}(s, e, \nu_s)$ for all $e\in E$ and almost all $s\in[0,T]$ by equicontinuity of $\nu\in\Pc_1(E)\mapsto\Hc^n \Phi_{\varphi}(s, e, \nu).$
    Then we can pass to the limit in the relation
	\begin{equation*}
     \E^{\Pt^n}\big[\big(\tilde\Mc^{\Phi_\varphi, n}_s-\tilde\Mc^{\Phi_\varphi, n}_t\big) h \big] = 0, \quad \text{for any $s\ge t$ and $h$ bounded continuous $\tilde{\Fc}_t$-measurable,}
	\end{equation*}
	This comes by similar arguments as~\cite[Lemma~4.12]{Claisse Ren Tan 2019} using the fact that there exists $C>0$ such that for all $n\in\N$ and $s\in[0,T],$
	\begin{equation*}
	\big| \Hc^n \Phi_{\varphi}(s, \Zt_s, \nu^n_s) \big|\le C \sup_{0\le t\le T} \sum_{k\in \Kt_t}\big( 1 + |\Xt_t^k| \big).
	\end{equation*}
	Together with Lemma~\ref{lemm:upperbound-bis}, it ensures the uniform integrability of the integral term in~\eqref{eq:martingale}.
   
   \noindent\emph{Step 4.} It remains to show that the solution to the limit martingale problem is unique given by $\P\circ Z^{-1}.$ Since $\Pt$ solves the martingale problem described in Step 3, it follows by arguments similar to~\cite[Proposition~2]{Claisse Kang Tan 2024} based upon Kurtz~\cite[Theorem 2.3]{Kurtz 2010} that we can construct on the (possibly enlarged) canonical space $\tilde \Om$ equipped with $\Pt,$ mutually independent Brownian motions $(\tilde W^k)_{k \in \K}$ and Poisson random measure $(\tilde Q^k)_{k \in \K}$,
	such that the canonical process $\Zt$ solves the analogous to SDE~\eqref{eq:SDE_McKean-VlasovB-bis} in this space.
   	Now it was established by \cite[Theorem~1]{Claisse Kang Tan 2024} that pathwise uniqueness holds for this class of SDEs.
   	Thus, by a Yamada-Watanabe-like theorem (see, \eg, Kurtz~\cite[Theorem~1.5]{Kurtz}), we deduce that uniqueness in distribution also holds and thus $\tilde \P$  coincides with $\P \circ Z^{-1}.$
\end{proof}

The stability property above also induces the convergence of the marginal measure $\mu^n_t=\pi(\nu^n_t)$ with $\pi$ defined by~\eqref{definition h}. This is what we really need in the proof of the invariance principle in Section~\ref{sec:marginal}. 

\begin{corollary}\label{cor:stability}
 Under the assumptions of Proposition~\ref{prop:stability},  the sequence of marginal measure $\mu^n_t = \pi(\P\circ (Z^n_t)^{-1})$ converges weakly to $\mu_t=\pi(\P\circ Z_t^{-1})$ in $\Mc(\R^d)$ for almost all $t\in [0,T].$
\end{corollary}

\begin{proof}
By Step 2 in the proof of Proposition~\ref{prop:stability}, it comes that the marginal $\P\circ (Z^{n}_t)^{-1}$ converges to $\P\circ Z_t^{-1}$ in $\Pc_1(E)$ for almost all $t\in[0,T].$ We can then conclude by continuity of the projection $\pi$ 
established in Lemma~\ref{lemma continuity}.
\end{proof}

\begin{remark} \rm
 Another way to enforce the convergence of $\mu^n_t$ to $\mu_t$ as in Corollary~\ref{cor:stability} is to replace Assumptions~(i) and (ii) in Proposition~\ref{prop:stability} as follows: there exists $\varepsilon>0$ such that 
 \begin{equation*}
 \sup_{n\in\N} \Big\|\sum_{\ell\ge 0}\ell^{1+\varepsilon} p_\ell^n\Big\|_{\infty}<+\infty \quad \text{and} \quad \sup_{n\in\N}\E\big[\<\xi_n, 1\>^{1+\varepsilon}\big]<+\infty.
 \end{equation*}
 Then we can recover the desired uniform integrability in Step 2 of the proof by an extension of~\eqref{ineq:supK-bis} stating that
 \begin{equation*}
 \E\Big[\big(\# K^n_t\big)^{1+\varepsilon}\Big] \le \E\big[\<\xi_n, 1\>^{1+\varepsilon}\big] e^{C T}, \quad \text{with } C = (1+\varepsilon) 2^{\varepsilon+1} \gammab \Big\|\sum_{\ell\ge 0}\ell^{1+\varepsilon} p_\ell^n\Big\|_{\infty}.
 \end{equation*}
\end{remark}

\end{document}